\documentclass[12pt,psamsfonts,leqno,oneside,letterpaper]{amsart}
\usepackage[dvips,text={6.5truein,9truein},left=1truein,top=1truein]{geometry}
\usepackage{amssymb,amsmath,amscd,enumerate}
\usepackage[pdftex]{graphicx}
\usepackage{tikz}
\usepackage{url}

\usepackage[colorlinks,linkcolor=blue,citecolor=blue,pdfstartview=FitH]{hyperref}
\input xy
\xyoption{all}
\SelectTips{cm}{12}

\usepackage{color}

\theoremstyle{definition}
\newtheorem{para}{}[section]

\newtheorem{remark}[para]{Remark}
\newtheorem{remarks}[para]{Remarks}
\newtheorem{notation}[para]{Notation}
\newtheorem{convention}[para]{Convention}
\newtheorem{definition}[para]{Definition}
\newtheorem{definitions}[para]{Definitions}
\newtheorem{claim}[equation]{Claim}
\newtheorem{fact}[para]{Fact}

\newtheorem{example}[para]{Example}

\newcommand\Alternatives{\begin{enumerate}[(i)]}
\newcommand\EndAlternatives{\end{enumerate}}
\newcommand\Conditions{\begin{enumerate}[(1)]}
\newcommand\EndConditions{\end{enumerate}}

\theoremstyle{plain}
\newtheorem{theorem}[para]{Theorem}
\newtheorem{lemma}[para]{Lemma}
\newtheorem{proposition}[para]{Proposition}
\newtheorem{corollary}[para]{Corollary}
\newtheorem{conjecture}[para]{Conjecture}

\numberwithin{equation}{para}
\numberwithin{figure}{section}

\newcommand\Number{\begin{para}}
\newcommand\EndNumber{\end{para}}
\newcommand\Definition{\begin{definition}}
\newcommand\EndDefinition{\end{definition}}
\newcommand\Definitions{\begin{definitions}}
\newcommand\EndDefinitions{\end{definitions}}
\newcommand\Theorem{\begin{theorem}}
\newcommand\EndTheorem{\end{theorem}}
\newcommand\Conjecture{\begin{conjecture}}
\newcommand\EndConjecture{\end{conjecture}}
\newcommand\Remark{\begin{remark}}
\newcommand\EndRemark{\end{remark}}
\newcommand\Remarks{\begin{remarks}}
\newcommand\EndRemarks{\end{remarks}}
\newcommand\Convention{\begin{convention}}
\newcommand\EndConvention{\end{convention}}
\newcommand\Notation{\begin{notation}}
\newcommand\EndNotation{\end{notation}}
\newcommand\Lemma{\begin{lemma}}
\newcommand\EndLemma{\end{lemma}}
\newcommand\Proposition{\begin{proposition}}
\newcommand\EndProposition{\end{proposition}}
\newcommand\Corollary{\begin{corollary}}
\newcommand\EndCorollary{\end{corollary}}
\newcommand\Claim{\begin{claim}}
\newcommand\EndClaim{\end{claim}}
\newcommand\Proof{\begin{proof}}
\newcommand\EndProof{\end{proof}}
\newcommand\Equation{\begin{equation}}
\newcommand\EndEquation{\end{equation}}

\newcommand\Bullets{\begin{itemize}}
\newcommand\EndBullets{\end{itemize}}

\renewcommand\Im{\mathop{\rm Im}}

\newcommand\calh{{\mathcal H}}

\newcommand\cals{\mathcal{S}}

\newcommand\co{\colon\thinspace}

\newcommand\bn{\mathbf{n}}
\newcommand\bo{\mathbf{0}}
\newcommand\bq{\mathbf{q}}
\newcommand\bs{\mathbf{s}}

\newcommand\bu{\mathbf{u}}
\newcommand\bv{\mathbf{v}}
\newcommand\bw{\mathbf{w}}
\newcommand\bx{\mathbf{x}}
\newcommand\by{\mathbf{y}}
\newcommand\bz{\mathbf{z}}

\begin{document}

\title{The Delaunay tessellation in hyperbolic space}

\author{Jason DeBlois}
\address{Department of Mathematics\\University of Pittsburgh}
\email{jdeblois@pitt.edu}

\begin{abstract}The Delaunay tessellation of a locally finite subset of the hyperbolic space $\mathbb{H}^n$ is constructed via convex hulls in $\mathbb{R}^{n+1}$.  For finite and lattice-invariant sets it is proven to be a polyhedral decomposition, and versions (necessarily modified from the Euclidean setting) of the empty circumspheres condition and geometric duality with the Voronoi tessellation are proved.  Some pathological examples of infinite, non lattice-invariant sets are exhibited.\end{abstract}

\maketitle

The main theorem of this paper describes a ``convex hull construction'' of canonical polyhedral decompositions with prescribed vertex set for arbitrary complete, finite-volume hyperbolic manifolds and locally finite subsets.  Various versions of this construction have been useful in the study of hyperbolic manifolds, beginning with work of Epstein--Penner in which the ``vertex set'' is essentially a collection of horospherical cusp neighborhoods \cite{EpstePen}.  Charney--Davis--Moussong gave a version for finite subsets of closed  hyperbolic manifolds \cite{CDM}, generalizing earlier work of N\"a\"at\"anen--Penner \cite{NaatPen}.  More recently, Cooper--Long translated the Epstein--Penner construction to the setting of convex projective manifolds \cite{CooLo}.

Our results are complementary to \cite{EpstePen} and generalize the main case of \cite{CDM}.  The main new case here, which is an important tool in our subsequent works \cite{DeB_locmax} and \cite{DeB_Voronoi}, covers finite subsets of finite-volume non-compact hyperbolic manifolds.  Compared to previous work this case exhibits substantial differences in the nature of the cells produced and of the decomposition's ``geometric duality'' relationship with the Voronoi tessellation.  As we will describe below the statement, the source of these differences also significantly complicates the proof.

We call our decomposition the ``Delaunay tessellation'' because it is characterized by an empty circumspheres condition, property (2) below.  It is constructed in Definition \ref{faces n stuff}.

\newcommand\PseudoEP{Let $\Gamma<\mathit{SO}^+(1,n)$ be a torsion-free lattice and $\cals$ a non-empty, locally finite, $\Gamma$-invariant set in $\mathbb{H}^n$.  The Delaunay tessellation of $\cals$ is a locally finite, $\Gamma$-invariant collection of convex polyhedra (the \mbox{\rm cells}) whose union is $\mathbb{H}^n$, satisfying:\begin{enumerate}
\item Each face of each cell is a cell, and distinct cells that intersect do so in a face of each; i.e. it is a \mbox{\rm polyhedral complex} in the sense of eg.~\cite[Dfn.~2.1.5]{DeLoRS}, with vertex set $\cals$.
\item For each metric ball or horoball $B$ of $\mathbb{H}^n$ that intersects $\cals$ but only on its boundary, ie.~such that $S = \partial B$ satisfies $B\cap\cals = S\cap\cals$, the closed convex hull of $S\cap\cals$ in $\mathbb{H}^n$ is a Delaunay cell contained in $B$.  Each Delaunay cell has this form.
\item For each parabolic fixed point $U$ of $\Gamma$ such that there is a horoball centered at $U$ and disjoint from $\cals$, there is a unique horosphere $S$ centered at $U$ such that the closed convex hull of $S\cap\cals$ in $\mathbb{H}^n$ is a $\Gamma_U$-invariant $n$-cell, where $\Gamma_U$ is the stabilizer of $U$ in $\Gamma$.  Each other cell is compact and has a metric circumsphere.\end{enumerate}
The Delaunay tessellation is uniquely determined by condition (2) above.}
\theoremstyle{plain}\newtheorem*{PEtheorem}{Theorem \ref{pseudo EP}}
\begin{PEtheorem}\PseudoEP\end{PEtheorem}

Here and in the remainder of the paper, we use the \textit{hyperboloid model} of hyperbolic space, $\mathbb{H}^n=\{-x_0^2+x_1^2+\hdots+x_n^2 = -1, x_0>0\}\subset\mathbb{R}^{n+1}$, inheriting a metric from the Lorentzian inner product on $\mathbb{R}^{n+1}$.  See Section \ref{intro Lorentzian} for details.  Its group of orientation-preserving isometries is $\mathit{SO}^+(1,n)<\mathit{GL}(\mathbb{R}^{n+1})$.  A \textit{lattice} in $\mathit{SO}^+(1,n)$ is a discrete subgroup with a finite-volume fundamental domain in $\mathbb{H}^n$.

Every complete, orientable hyperbolic $n$-manifold of finite volume is isometric to $\mathbb{H}^n/\Gamma$ for a torsion-free lattice $\Gamma<\mathit{SO}^+(1,n)$.   Given such a manifold $M$, to obtain a polygonal decomposition with prescribed locally finite vertex set $\cals_0\subset M$ we take the Delaunay tessellation of the preimage $\cals$ of $\cals_0$ in $\mathbb{H}^n$ and project cells to $M$.  See Corollary \ref{down below}.

The idea of the convex hull construction is to take the convex hull $\mathrm{Hull}(\cals)$ of $\cals$ in $\mathbb{R}^{n+1}$, and for each proper face $F$ let $r_n(F)$ be a Delaunay cell, where $r_n$ is the projection to $\mathbb{H}^n$ along rays through the origin.  Any support plane for $F$ intersects $\mathbb{H}^n$ in a circum(horo)sphere for $r_n(F)$.  If $\cals$ satisfies the hypotheses of any of \cite{CDM}, \cite{EpstePen} or \cite{NaatPen}, then each proper face $F$ of $\mathrm{Hull}(\cals)$ has a support plane parallel to a \textit{space-like} subspace of $\mathbb{R}^{n+1}$, on which the Lorenzian inner product has positive-definite restriction.  This  no longer holds in our setting.

In particular, if $M = \mathbb{H}^n/\Gamma$ is non-compact and $\cals_0\subset M$ is finite then by Corollary \ref{co-finite} every cusp of $M$ is contained in a \textit{horospherical} Delaunay cell, described in Theorem \ref{pseudo EP}(3).  The support plane of an $n$-dimensional face $F$ of $\mathrm{Hull}(\cals)$ such that $r_n(F)$ is horospherical is parallel to a \textit{light-like} subspace $V$ of $\mathbb{R}^{n+1}$, on which the Lorentzian inner product's restriction is degenerate.  The corresponding \textit{parabolic fixed point} of $\Gamma$ is the one-dimensional light-like subspace $U$ of $V$.  See Section \ref{margulis} for details.

In consequence, the decompositions produced by Theorem \ref{pseudo EP} are not ``Euclidean'', in the sense of \cite{CDM} and \cite{EpstePen}, when there are horospherical Delaunay cells.  In addition, affine hyperplanes parallel to space-like subspaces enjoy compactness and stability properties (which also hold in the projective setting of \cite{CooLo}, see the bullet spanning pp.~6--7 there) that those parallel to light-like planes do not.  Our proof of Theorem \ref{pseudo EP} roughly parallels those of the main theorems of \cite{CDM} and \cite{EpstePen}, but the steps that require perturbing support planes are significantly complicated by this fact.  Lemma \ref{rotate plane} is a key new technical result explicitly describing useful perturbations of support planes that are parallel to light-like subspaces.

In another departure from \cite{CDM} and \cite{EpstePen}, the \textit{geometric dual} to the \textit{Voronoi tessellation} of $\cals$ is a proper subcomplex of the Delaunay tessellation when there are horospherical Delaunay cells.  For definitions see Section \ref{into Voronoi}.  There we prove:

\newcommand\TheRGD{The \mbox{\rm geometric dual complex} of a locally finite set $\cals\subset\mathbb{H}^n$, consisting of Delaunay cells geometrically dual to Voronoi cells, is a polyhedral complex.  For each metric ball $B$ of $\mathbb{H}^n$ that intersects $\cals$ such that $B\cap\cals = S\cap\cals$, where $S=\partial B$, the closed convex hull of $S\cap\cals$ in $\mathbb{H}^n$ is a compact geometric dual cell.  Each geometric dual cell has this form.}
\newtheorem*{RGDTheorem}{Theorem \ref{the real geometric dual}}

\begin{RGDTheorem}\TheRGD\end{RGDTheorem}

Note that Theorem \ref{the real geometric dual} applies to arbitrary locally finite sets and does not require lattice-invariance.  In fact we only restrict to the lattice-invariant setting in Section \ref{manifolds}.  Sections \ref{intro Lorentzian} and \ref{convex project} establish general preliminaries, respectively on Lorentzian geometry and convex hulls in Euclidean space.  In section \ref{Delaunay existence} we consider the case that $\cals\subset\mathbb{H}^n$ is finite.  The main result of this section is Proposition \ref{the finite case}, which describes the Delaunay tessellation of such $\cals$.  

The case where $\cals$ is finite is of independent interest in computational geometry.  In particular, Devillers et.~al.~constructed the Delaunay tessellation of such $\cals$ in \cite{BDT} and the predecessor \cite{DMT}, using a different approach.   They exclude cells with non-metric circumspheres and so obtain the geometric dual subcomplex of our complex.  See Remark \ref{empty characterization}.  Proposition \ref{the finite case} can be regarded as giving a geometric meaning to the cells excluded by Devillers et.~al.  We also remark that computational geometers have known a convex hull construction of \textit{Euclidean} Delaunay tessellations for many years: in a survey paper of Graham--Yao \cite[p.~693]{GrahamYao}, the idea is attributed to a 1979 paper of K.Q.~Brown.

Section \ref{bad example} explores the case of infinite but not lattice-invariant sets $\cals$.  There we give examples to illustrate how some aspects of Theorem \ref{pseudo EP} can fail in this setting.  We note that Akiyoshi--Sakuma have considered infinite, non lattice-invariant sets in the light cone and encountered similar pathologies \cite{AkSak}.

\subsection*{Acknowledgment}  This paper was inspired by a talk at Pitt by CMU's Noel Walkington.  Many thanks to the referee for helpful comments, observations, and suggestions.

\section{A brief introduction to the hyperboloid model}\label{intro Lorentzian}  This section collects a handful of well known basic facts on hyperbolic geometry, for easy referencing and to establish notation.  The \textit{Lorentzian inner product} on $\mathbb{R}^{n+1}$ is given by
$$ (x_0,x_1,\hdots,x_n)\circ (y_0,y_1,\hdots,y_n) = -x_0y_0+x_1y_1+\hdots+x_ny_n $$
This restricts on the tangent bundle of the \textit{hyperboloid model} $\mathbb{H}^n = \{\bx\,|\,\bx\circ\bx=-1, x_0>0\}$ to a genuine Riemannian metric, the \textit{hyperbolic metric} with constant sectional curvature $-1$.  

For $0< k < n$, $\mathbb{R}^{k+1}\times \{(0,\hdots,0)\}$ inherits its own Lorentzian inner product from that of the ambient $\mathbb{R}^{n+1}$, so its intersection with $\mathbb{H}^n$ is an isometrically embedded copy of $\mathbb{H}^k$.  The isometry group $\mathit{SO}^+(1,n)$ (called $\mathit{PO}(1,n)$ in \cite{Ratcliffe}) of $\mathbb{H}^n$ acts transitively on the set of $(k+1)$-dimensional \textit{time-like} subspaces, those which intersect $\mathbb{H}^n$ \cite[Theorem 3.1.6]{Ratcliffe}.  Their intersections with $\mathbb{H}^n$ comprise its collection of $k$-dimensional totally geodesic subspaces.

In particular, the geodesic containing distinct $\bx$ and $\by$ in $\mathbb{H}^n$ is its intersection with  the $2$-dimensional subspace of $\mathbb{R}^{n+1}$ that they span (see the Definition on p.~64 of \cite{Ratcliffe}, and Corollary 4 in the same section).  It follows from \cite[Theorem 3.2.5]{Ratcliffe} that the unique geodesic ray of $\mathbb{H}^n$ from $\bx$ through $\by$ is explicitly parametrized by arclength as:\begin{align}\label{geodesic ray}
 t\mapsto \cosh t\,\bx+\sinh t\, \bn,\end{align}
where $\bn$ is the unit vector in the direction of the component $\by + (\bx\circ\by)\bx$ of $\by$ normal to $\bx$.  The ray from $\bx$ runs through $\by$ at $t = \cosh^{-1}(-\bx\circ\by)$, so this is the distance $d_H(\bx,\by)$.

Note that for $\bx$ and $\by$ with $x_0 >0$, $y_0>0$, $\bx\circ\bx = a\leq0$ and $\by\circ\by = b\leq0$:\begin{align}\label{x circ y}
 \bx\circ\by +\sqrt{ab} & = -x_0y_0 + x_1y_1+\hdots x_ny_n + \sqrt{ab}\nonumber \\
   & = -\sqrt{-a+x_1^2+\hdots x_n^2}\sqrt{-b+y_1^2+\hdots y_n^2} + \sqrt{ab} + x_1y_1+\hdots x_ny_n
   \leq 0 \end{align}
Thus $\bx\circ\by\leq-\sqrt{ab}$.  The inequality above is the Cauchy-Schwartz inequality, which further implies strict inequality if and only if $\by$ is not a scalar multiple of $\bx$.

\begin{lemma}\label{normal vectors}  For each $n$-dimensional subspace $V$ of $\mathbb{R}^{n+1}$ there is a unique $1$-dimensional subspace $V^{\perp}$ such that $\bv\circ\bu = 0$ for each $\bv\in V$ and $\bu\in V^{\perp}$.\begin{itemize}
  \item  If $V$ is space-like then $V^{\perp}$ is time-like.
  \item  If $V$ is light-like then $V^{\perp} = V\cap L$, where $L = \{\bu\,|\,\bu\circ\bu=0\}$ is the \mbox{\rm light cone}.
  \item  If $V$ is time-like then $V^{\perp}$ is space-like.\end{itemize}
We will say that any $\bu\neq\bo$ in $V^{\perp}$ is a \mbox{\rm normal} to $V$.\end{lemma}

Here, following common usage, we call a subspace $V$ \textit{space-like}, \textit{light-like}, or \textit{time-like} if the Lorentzian inner product's restriction to $V$ is respectively positive-definite, positive-semidefinite, or neither.  A single vector is space-like, light-like, or time-like if its span is.

Lemma \ref{normal vectors} is standard and is taken for granted in eg.~\cite{EpstePen} and \cite{CDM}.  Here is a proof sketch.  We note first that that $\{0\}\times\mathbb{R}^n$ is space-like, and its Lorentz-orthogonal complement $\mathbb{R}\times\{\bo\}$ is time-like.  For any other $n$-dimensional subspace $V$, $V_0 = V\cap \left(\{0\}\times\mathbb{R}^n\right)$ is  space-like, so the usual Gram-Schmidt process produces an orthonormal basis for $V_0$.  $V_0$ has codimension one in $V$, so the addition of a vector $\bv_0\in V-V_0$ fills this out to a basis for $V$.  We may take $\bv_0$ orthogonal to $V_0$ by subtracting off its orthogonal projection (defined in the usual way).

The hypothesis that $V$ is space-like, light-like or time-like now forces $\bv_0\circ\bv_0$ to be respectively positive, zero, or negative.  Note in the second case that $\bu=\bv_0\in V^{\perp}$ is a normal.  In the first we may normalize so that $\bv_0\circ\bv_0 = 1$, define an orthogonal projection $\mathbb{R}^{n+1}\to V$ in the usual way, and subtract from any $\bx\in\mathbb{R}^{n+1}-V$ its projection to $V$ to produce a normal $\bu$.  In the final, time-like case we normalize so that $\bv_0\circ\bv_0 =-1$ and define $p_V\co\mathbb{R}^{n+1}\to V$ by\begin{align}
  \label{time-like projection}  p_V(\bx) = p_{V_0}(\bx)-(\bx\circ\bv_0)\bv_0,\end{align}
where $p_{V_0}$ is the projection to $V_0$.  Then for any $\bx\in \mathbb{R}^{n+1}-V$, $\bu=\bx-p_V(\bx)$ is a normal to $V$.  Sylvester's law of inertia implies in this case that $\bu$ is space-like, and if $V$ is space-like that any normal is time-like.  Non-degeneracy of the Lorentzian inner product implies (in all cases) that $V^{\perp}$ is one-dimensional.

We will call \textit{hyperspheres} the non-empty intersections between $\mathbb{H}^n$ and $n$-dimensional affine subspaces of $\mathbb{R}^{n+1}$.  They are classified below.

\begin{lemma}\label{classification of spheres}  For an $n$-dimensional subspace $V$ of $\mathbb{R}^{n+1}$ and $\bx_0\notin V$ such that $P =V+\bx_0$ intersects $\mathbb{H}^n$, the hypersphere $S\doteq P\cap\mathbb{H}^n$ is classified as follows:\begin{itemize}
\item If $V$ is space-like then $P = \{\bx\,|\,\bx\circ\bu = \bx_0\circ\bu\}$, where $\bu$ is the unique normal to $V$ in $\mathbb{H}^n$, and $S=\{\bx\in\mathbb{H}^n\,|\,\cosh d_H(\bx,\bu) = -\bx_0\circ\bu\}$ is a metric sphere centered at $\bu$.
\item If $V$ is light-like then $P=\{\bx\,|\,\bx\circ\bu=-1\}$, where $\bu$ is the unique normal to $V$ with $\bx_0\circ\bu=-1$, and $S=\{\bx\in\mathbb{H}^n\,|\,\bx\circ\bu = -1\}$ is the \mbox{\rm horosphere centered at $\bu$}.
\item If $V$ is time-like then for $\lambda = \bx_0\circ\bu\in V$, where $\bu$ is a unit normal to $V$, $S$ is a component of the equidistant locus $\left\{\bx\in\mathbb{H}^n\,|\,\cosh d_H(\bx,V\cap\mathbb{H}^n) = \sqrt{1+\lambda^2}\right\}$.\end{itemize}
Every metric sphere, horosphere, or equidistant to a geodesic subspace is of the form above.\end{lemma}

For now we will take as a definition that a \textit{horosphere} of $\mathbb{H}^n$ is a set of the form $\{\bx\in\mathbb{H}^n\,|\,\bx\circ\bu = -1\}$ for some fixed $\bu$ with $\bu\circ\bu=0$ and $u_0>0$.  Its \textit{ideal point} is $U=\mathrm{span}(\bu)$.

Lemma \ref{classification of spheres} above is a more-detailed version of Lemma 2 of \cite{CDM}, but the additional detail in the statement here is largely obvious or contained in the proof there.  We do fill in some detail in the case that $V$ is time-like, however.  In this case, for a unit normal $\bu$ to $V$ and $p_V\co \mathbb{R}^{n+1}\to V$ as defined in (\ref{time-like projection}), any $\bx\in\mathbb{R}^{n+1}$ decomposes as $p_V(\bx)+\lambda\bu$ for some $\lambda\in\mathbb{R}$.  Then $\bx\circ\bx = p_V(\bx)\circ p_V(\bx) + \lambda^2$, so for any $\bx\in\mathbb{H}^n$ and $\by\in V\cap\mathbb{H}^n$ we have:
$$ \bx\circ\by = p_V(\bx)\circ\by \leq -\sqrt{1+\lambda^2} $$
The inequality here is an application of (\ref{x circ y}).  Equality is attained if $\by$ is a scalar multiple of $p_V(\bx)$ (see below (\ref{x circ y})), so the distance from $\bx$ to $V\cap\mathbb{H}^n$ is $\cosh^{-1}\sqrt{1+\lambda^2}$.

Now note for any $\bx\in P = V+\bx_0$ that since $\bx-\bx_0\in V$ we have $\bx - p_V(\bx) = \bx_0 - p_V(\bx_0)$, so $\lambda$ as above is $\bx_0\circ\bu$, and the distance from $S = P\cap\mathbb{H}^n$ to $V\cap\mathbb{H}^n$ is constant.  The two components of the equidistant locus in question are $P\cap\mathbb{H}^n$ and $(V-\bx_0)\cap\mathbb{H}^n$.


\begin{figure}
\begin{tikzpicture}

\begin{scope}[xshift=-5cm]
\draw (-2,0) -- (2,0);
\node [above right] at (-2,0) {$\mathbb{R}$};

\draw (0,1.1) circle [radius=0.9];
\fill [color=gray] (0,0.7) circle [radius=0.1];
\fill (0.9,1.1) circle [radius=0.1];
\node [right] at (0.9,1.1) {$y$};
\fill (-0.9,1.1) circle [radius=0.1];
\node [left] at (-0.9,1.1) {$z$};
\fill (0,2) circle [radius=0.1];
\node [above] at (0,2) {$x$};
\end{scope}

\begin{scope}
\draw (-2,0) -- (2,0);

\draw (0,1) circle [radius=1];
\fill [color=white] (0,0) circle [radius=0.1];
\draw [very thick, color=gray] (0,0) circle [radius=0.1];
\fill (0.995,1.1) circle [radius=0.1];
\node [right] at (0.995,1.1) {$y$};
\fill (-0.995,1.1) circle [radius=0.1];
\node [left] at (-0.995,1.1) {$z$};
\fill (0,2) circle [radius=0.1];
\node [above] at (0,2) {$x$};
\end{scope}

\begin{scope}[xshift=5cm]
\draw (-2,0) -- (2,0);

\draw (0,0.9) circle [radius=1.1];
\draw [very thick, color=gray] (0.632,0) arc (0:180:0.632);
\fill (1.08,1.1) circle [radius=0.1];
\node [right] at (1.08,1.1) {$y$};
\fill (-1.08,1.1) circle [radius=0.1];
\node [left] at (-1.08,1.1) {$z$};
\fill (0,2) circle [radius=0.1];
\node [above] at (0,2) {$x$};
\end{scope}

\end{tikzpicture}
\caption{Hyperspheres of $\mathbb{H}^2$ determined by three points.}
\label{three pt circ}
\end{figure}
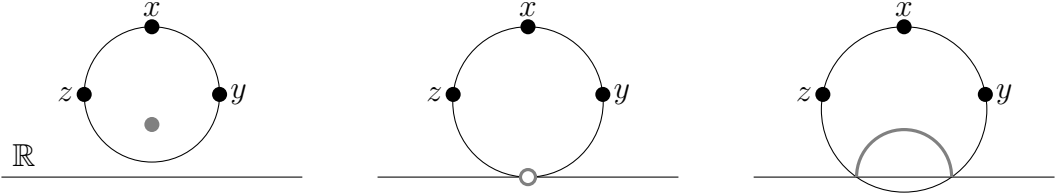

\begin{example}\label{three pt sphere}  The three possibilities for the hypersphere $S$ containing $\bx$, $\by$ and $\bz$ in $\mathbb{H}^2$ are pictured in Figure \ref{three pt circ} in the \textit{upper half-plane model} $\{z\in\mathbb{C}\,|\,\Im\,z>0\}$ (with Riemannian metric $g_z = \frac{g_E}{\Im z}$, where $g_E$ is the Euclidean metric).  This model is a useful visual aid.  Its geodesics are lines and semicircles perpendicular to $\mathbb{R}$; its hyperspheres are its intersections with Euclidean circles and straight lines.

In the figure we may take $d_H(x,y)$ and $d_H(x,z)$ to be fixed, with $d_H(y,z)$ increasing from left to right.  In each case $S$ is the intersection with $\mathbb{H}^2$ of the Euclidean circle in $\mathbb{C}$ containing $\bx$, $\by$ and $\bz$.  On the left this circle is contained in $\mathbb{H}^2$ and so is also a hyperbolic circle, though with a different center (the grey dot).  In the middle this circle is tangent to $\mathbb{R}$ and $S$ is a horosphere, the complement of the point of intersection.  On the right, $S$ is a component of the equidistant locus to the hyperbolic geodesic pictured in grey, which joins the points of intersection between $\mathbb{R}$ and the Euclidean circumcircle for $\bx$, $\by$ and $\bz$.\end{example}

Each hypersphere $S = P\cap\mathbb{H}^n$ bounds two closed regions, each the intersection of $\mathbb{H}^n$ with a \textit{half-space}: the closure of a component of $\mathbb{R}^{n+1} - P$.  Below we say a set in $\mathbb{H}^n$ is \textit{convex} if it contains the unique geodesic arc between any two of its points.

\begin{lemma}\label{convex component}  For an $n$-dimensional subspace $V$ of $\mathbb{R}^{n+1}$ and $\bx_0\notin V$ such that $P = V+\bx_0$ intersects $\mathbb{H}^n$, let $B$ be the half-space containing $\bo$ with $\partial B = P$.  Then $B\cap\mathbb{H}^n$ is the unique closed, convex region in $\mathbb{H}^n$ with boundary the hypersphere $S = P\cap\mathbb{H}^n$.  Furthermore:\begin{itemize}
\item  If $V$ is space-like then $B = \{\bx\,|\,\bx\circ\bu \geq \bx_0\circ \bu\}$, where $\bu\in\mathbb{H}^n\cap V^{\perp}$, and $B\cap\mathbb{H}^n$ is the metric ball of radius $\cosh^{-1} -\bx_0\circ\bu$ centered at $\bu$.  Also, $B\cap U_{n+1}$ is compact, where $U_{n+1}=\{\bx\,|\,\bx\circ\bx\leq -1, x_0 \geq 0\}$ is the convex hull of $\mathbb{H}^n$ in $\mathbb{R}^{n+1}$.
\item If $V$ is light-like then $B = \{\bx\,|\,\bx\circ\bu\geq -1\}$, where $\bu$ is the normal to $V$ such that $\bu\circ\bx_0=-1$.  In this case we say $B\cap\mathbb{H}^n$ is the \mbox{\rm horoball centered at $\bu$}.  It contains the \mbox{\rm geodesic ray $\gamma_{\bx}$ from $\bx$ in the direction of $\bu$}, given by: 
$$\gamma_{\bx}(t)= (\cosh t-\sinh t)\bx + \sinh t \bu = e^{-t}\bx+\sinh t\bu,\quad\mbox{for}\ t\geq 0$$ 
\item If $V$ is time-like then $B = \{\bx\,|\,\bx\circ\bu\geq\bx_0\circ\bu\}$, where $\bu$ is the unit normal to $V$ with $\bx_0\circ\bu <0$.  In this case $V\subset B$.
\end{itemize}
Also, if $V$ is time-like then each half-space bounded by $V$ has convex intersection with $\mathbb{H}^n$. For such a half-space $B$, we say $B\cap\mathbb{H}^n$ is a \mbox{\rm hyperbolic half-space} bounded by $V\cap\mathbb{H}^n$.\end{lemma}

\begin{proof}  For $\bx,\by\in\mathbb{H}^n$ and $t\in\mathbb{R}$, call $\bz(t)$ the point on the geodesic from $\bx$ through $\by$ defined by the formula of (\ref{geodesic ray}).  Chasing the definition there, for any $\bu\in\mathbb{R}^{n+1}$ we have:\begin{align}\label{z circ u}
  \phi(t) \doteq \bz(t)\circ\bu = (\cosh t + \sinh t(\bx\circ\by)/n)(\bx\circ\bu) + \sinh t (\by\circ\bu)/n,
\end{align}
where $n = \sqrt{(\bx\circ\by)^2-1}$.  The coefficients of $\bx\circ\bu$ and $\by\circ\bu$ in (\ref{z circ u}) are non-negative on the interval $[0,\cosh^{-1}-\bx\circ\by]$, and $\phi''(t) = \phi(t)$.  Therefore if $\bx\circ\bu$ and $\by\circ\bu$ are both non-positive then $\phi$ is concave on this interval, so its values here are at least $\min\{\bx\circ\bu,\by\circ\bu\}$ (its values at the endpoints).  In particular, in the three cases above where $\bx_0\notin V$, for any $\bx$ and $\by$ in $S = (V+\bx_0)\cap\mathbb{H}^n$ the geodesic from $\bx$ to $\by$ lies in $B = \{\bx\,|\,\bx\circ\bu\geq\bx_0\circ\bu\}$ since in each case $\bx_0\circ\bu <0$.  It follows that the half-space bounded by $P$ opposite $B$ does not have convex intersection with $\mathbb{H}^n$.  

In the case that $V$ is space-like, for its normal $\bu$ in $\mathbb{H}^n$ every $\bx\in\mathbb{H}^n$ satisfies $\bx\circ\bu < 0$ by (\ref{x circ y}).  The above thus implies that $B\cap\mathbb{H}^n$ is convex.  It is a metric ball by definition of $d_H$ (see below \ref{geodesic ray}).  It is an exercise to show that $B\cap U_{n+1}$ is compact when $V = \{0\}\times\mathbb{R}^{n+1}$ (so $\bu = (1,0,\hdots,0)$ and $P = V+\lambda\bu$ for some $\lambda\geq 1$); transitivity of the $\mathit{SO}^+(1,n)$-action on $\mathbb{H}^n$, hence on space-like hyperplanes, now implies the same for arbitrary space-like $V$.

If $V$ is light-like then for our choice of light-like normal $\bu$, since $\bu\circ\bx =-1$ for some $\bx\in\mathbb{H}^n$ the Cauchy-Schwarz inequality implies that $u_0 > 0$.  Hence again every $\bx\in\mathbb{H}^n$ satisfies $\bx\circ\bu<0$ by (\ref{x circ y}), and $B\cap\mathbb{H}^n$ is convex.  By a simple direct computation, $\gamma_{\bx}(t)\in B\cap\mathbb{H}^n$ for all $t\geq 0$.

For time-like $V$ with a unit normal $\bu$ we first note that since the half-spaces bounded by $V$ are characterized by the inequalities $\bx\circ\bu\geq 0$ and $\bx\circ\bu\leq0$, for $\bx$ and $\by$ in one or the other the function $\phi(t)$ defined in (\ref{z circ u}) is respectively convex or concave on $[0,\cosh^{-1} -\bx\circ\by]$.  It thus follows as above that each hyperbolic half-space bounded by $V\cap\mathbb{H}^n$ is convex.

Now to show that $B\cap\mathbb{H}^n$ is convex for some half-space $B = \{\bx\,|\,\bx\circ\bu\geq\bx_0\circ\bu\}$, where $\bx_0\notin V$ and $\bx_0\circ\bu<0$ we consider three cases for $\bx,\by\in B\cap\mathbb{H}^n$: that each of $\bx\circ\bu$ and $\bx\circ\bu$ is less than $0$, that each is at least $0$, or neither.  In the first two cases we have already shown above that the geodesic from $\bx$ to $\by$ is in $B\cap\mathbb{H}^n$.  In the final case, assuming that (say) $\bx\circ\bu < 0 \leq \by\circ\bu$ we have $\phi'(t) \geq 0$ for all $t$, since the coefficient of $\bx\circ\bu$ in (\ref{z circ u}) is decreasing.  Again the result follows.\end{proof}

We finally record a weak Hahn--Banach theorem (see eg.~\cite[Th.~11.4.1]{Berger}) for hyperbolic space.

\begin{lemma}\label{HB}  For a closed convex set $C\subset\mathbb{H}^n$ and $\bx_0\notin C$, there is an $n$-dimensional time-like subspace $V\subset\mathbb{R}^{n+1}$ with $C$ and $\bx_0$ in opposite half-spaces bounded by $V$.\end{lemma}

\begin{proof}  Let $\bx$ be a closest point in $C$ to $\bx_0$, and let $\bu=\bx_0 + (\bx\circ\bx_0)\bx$ be the component of $\bx_0$ normal to $\bx$.  Then $\bu\circ\bu = -1+(\bx\circ\bx_0)^2$ is greater than $0$ by (\ref{x circ y}), so $V=\bu^{\perp}$ is a time-like subspace through $\bx$.  Moreover, $\bu\circ\bx_0 = \bu\circ\bu>0$.

For $\by\in C-\{\bx\}$, let $\bz(t)$ be the geodesic joining $\bx$ to $\by$ described in (\ref{geodesic ray}).  Then $\phi(t) = \bz(t)\circ\bx_0$ has non-positive derivative at $t=0$, since $\gamma(t)\in C$ for all $t\in[0,\cosh^{-1} -\bx\circ\by]$ but $\bx$ is closest in $C$ to $\bx_0$.  Applying this fact to (\ref{z circ u}), with $\bx_0$ replacing $\bu$ there, we find that $(\bx\circ\by)(\bx\circ\bx_0)+(\by\circ\bx_0) \leq 0$.  It follows that $\bu\circ\by \leq 0$.\end{proof}

\section{Convex sets and projection to the hyperboloid}\label{convex project}
A subset of $\mathbb{R}^n$ is \textit{convex} if it contains the line segment joining any two of its points.  The \textit{convex hull} of $\cals\subset\mathbb{R}^n$ is the minimal (with respect to inclusion) convex set containing $\cals$, and the \textit{closed convex hull} is the minimal closed, convex set containing $\cals$.  We will denote the closed convex hull of $\cals$ by $\mathrm{Hull}(\cals)$.  It is the closure of the convex hull (cf.~\cite[11.2.3]{Berger}).

It is not hard to show that $U_{n+1}= \{\bx\in\mathbb{R}^{n+1}\,|\,\bx\circ\bx \leq -1,\ x_0 > 0\}$ is a closed, strictly convex subset of $\mathbb{R}^{n+1}$ bounded by $\mathbb{H}^n$, so for $\cals\subset\mathbb{H}^n$, $\mathrm{Hull}(\cals)\subset U_{n+1}$.  We will construct the Delaunay tessellation of $\cals$ by projecting ``visible'' faces of $\mathrm{Hull}(\cals)$ (this notation comes from \cite{AkSak}) to $\mathbb{H}^n$ along rays from the origin.

\begin{definition}  For a closed, convex subset $C$ of $\mathbb{R}^{n+1}$ not containing $\bo$, say $\bx\in C$ is \textit{visible} if it is the first point of intersection between $C$ and the ray from $\bo$ through $\bx$.\end{definition}

\begin{lemma}\label{convex hull}  Let $U_{n+1} = \{\bx\in\mathbb{R}^{n+1}\,|\,\bx\circ\bx \leq -1,\ x_0 > 0\}$, and let $r_n\co U_{n+1}\to\mathbb{H}^n$ take $\bx$ to $\bx/\sqrt{-\bx\circ\bx}$.  For $\cals\subset\mathbb{H}^n$, the set of visible points of $\mathrm{Hull}(\cals)$ projects bijectively under $r_n$ to a convex subset of $\mathbb{H}^n$ containing $\cals$ and contained in the closed convex hull of $\cals$ in $\mathbb{H}^n$.\end{lemma}

\begin{proof}We claim that $r_n$ maps $\mathrm{Hull}(\cals)$ onto a set as above.  The lemma will follow, since it is clear from the definition that $\mathrm{Hull}(\cals)$ and its set of visible points have the same image, and that $r_n$ is injective on the set of visible points.


The image of $\mathrm{Hull}(\cals)$ is convex.  The key fact is that for any $\bx_0,\by_0\in U_{n+1}$ the line segment $[\bx,\by]$ in $\mathbb{R}^{n+1}$ that joins them projects under $r_n$ to the geodesic arc joining $\bx=r_n(\bx_0)$ to $\by=r_n(\by_0)$ in $\mathbb{H}^n$.  This follows from (\ref{geodesic ray}) and a brief calculation that $\cosh t\bx+\sinh t\bn$ is a positive linear combination of $\bx$ and $\by$ for $t\in(0,\cosh^{-1}(-\bx\circ\by))$.

For any $\bx_0$ outside the closed convex hull of $\cals$ in $\mathbb{H}^n$, Lemma \ref{HB} supplies a time-like subspace $V$ separating $\bx_0$ from $\cals$, hence also from $\mathrm{Hull}(\cals)$.  Since $V$ is a subspace $r_n(V\cap U_{n+1}) = V\cap\mathbb{H}^n$, and it follows that $V\cap\mathbb{H}^n$ separates $\bx_0$ from $r_n(\mathrm{Hull}(\cals)$.\end{proof}


In general $r_n(\mathrm{Hull}(\cals))$ may not be closed; see Section \ref{bad example}.  Lemma \ref{convex hull boundary} below describes such cases, but proving it requires more preliminaries on convex subsets of $\mathbb{R}^n$.

\begin{definition} An $n$-dimensional affine plane $P$ is a \textit{support plane} for a closed convex subset $C$ of $\mathbb{R}^{n+1}$ if $C$ is contained in one of the half-spaces bounded by $P$, and $P\cap C\neq\emptyset$.

For a $n$-dimensional affine plane $P$ of $\mathbb{R}^{n+1}$, $\bx_0\in P$, and a half-space $B$ bounded by $P$, a vector $\eta$ is a \textit{Euclidean outward normal to $B$} if $\eta\cdot(\bx-\bx_0)\leq 0$ for all $\bx\in B$.  (Here ``$\cdot$'' is the Euclidean inner product.)\end{definition}

\begin{remark}\label{Euclid vs Lorentz}  For $\eta = (e_0,e_1,\hdots,e_n)$, $\bar{\eta} \doteq (e_0,-e_1,\hdots,-e_n)$ shares an initial entry with $\eta$ and satisfies $\bar{\eta}\circ\bv=-\eta\cdot \bv$ for any $\bv\in\mathbb{R}^{n+1}$.  In particular, if $P$ is parallel to a space-like or light-like subspace and intersects $\mathbb{H}^n$, and $\eta$ is a Euclidean outward normal to the half-space $B$ with $B\cap\mathbb{H}^n$ convex then $B = \{\bx\,|\,\bx\circ\bar{\eta}\geq  \bx_0\circ\bar{\eta}\}$ for any $\bx_0\in P$ (compare Lemma \ref{convex component}).\end{remark}

\begin{lemma}\label{plane convergence}  For a closed convex set $C\subset\mathbb{R}^{n+1}$, say a sequence of support planes $P_n$ for $C$ \mbox{\rm converges} to a plane $P$ if there exist $\bx\in P$ and a Euclidean normal $\eta$ to $P$ approached by a sequence $\{(\bx_n,\eta_n)\}$, where $\bx_n\in P_n$ and $\eta_n$ is a Euclidean outward normal to a half-space bounded by $P_n$ and containing $C$.  Then $\bx\in C$, $P$ is a support plane for $C$ through $\bx$, and $\eta$ is a Euclidean outward normal to a half-space $B$ bounded by $P$ and containing $C$.\end{lemma}

\begin{proof}  Since $C$ is closed, $\bx\in C$.  The lemma follows upon fixing any $\by\in C$ and taking a limit of the inequality $\eta_n\cdot(\by-\bx_n)\leq0$.\end{proof}

For any closed, convex set $C$ and $\bx\in\partial C$ there is a support plane for $C$ through $\bx$ \cite[Prop.~11.5.2]{Berger}.  If $\bx$ is visible, more can be said.

\begin{lemma}\label{visible point}  Let $C$ be a closed, convex subset of $\mathbb{R}^{n+1}$ such that $\bo\notin C$.  For every visible point $\bx\in C$ there is a support plane $P$ for $C$ through $\bx$ that separates $C$ from $\bo$; i.e. such that the half-space $B$ bounded by $P$ with $B\cap C=P\cap C$ contains $\bo$.\end{lemma}

\begin{proof}  We will assume $C$ has a non-empty interior, for if not it is entirely contained in a supporting hyperplane \cite[Prop.~11.2.7]{Berger}.  Thus every support plane $P$ for $C$ bounds a unique half-space $B_0$ containing $C$.  The lemma holds if and only if $\eta\cdot\bx\leq 0$ for some such $P$ and a Euclidean outward normal $\eta$ to $B_0$, since $\eta\cdot\bx = - \eta\cdot(\bo-\bx)$.

If the ray from $\bo$ through $\bx$ contains $t\bx$ for some $t>1$, then $\bu\cdot\bx\leq 0$ for any outward normal $\bu$ to a support plane through $\bx$ (since $(t-1)\bx\cdot\bu\leq 0$).  For $\bx$ without this property, choose a sequence $\{\by_n\}$ of points in the interior of $C$ approaching $\bx$, and let $\{\bx_n\}$ be the corresponding sequence with $\bx_n$ the visible point on the ray from $\bo$ through $\by_n$.  Passing to a subsequence, we may assume $\{\bx_n\}$ also converges to $\by$.

For each $n$ let $\eta_n$ be an outward normal with length $1$ to a support plane for $C$ through $\bx_n$.  Passing to a subsequence again, we assume $\eta_n$ converges to a vector $\eta$.  Then $\eta\cdot\bx = \lim_{n\to\infty} \eta_n\cdot\bx_n \leq 0$ and for any $\by\in C$, $\eta\cdot(\by-\bx) = \lim_{n\to\infty} \eta_n\cdot(\by-\bx_n)\leq 0$.\end{proof}

\begin{lemma}\label{ratatouille}  If $P$ is an $n$-dimensional affine subspace of $\mathbb{R}^{n+1}$ with $P\cap U_{n+1}\neq\emptyset$ then $r_n(P\cap U_{n+1})\subset B\cap\mathbb{H}^n$ for a half-space $B$ bounded by $P$ with $\bo\in B$ and $B\cap\mathbb{H}^n$ convex.\end{lemma}

\begin{proof}  If $P$ contains $\bo$ then so do both half-spaces bounded by $P$.  In this case $r_n(P\cap U_{n+1})\subset P$ is contained in each such half-space.  We therefore suppose that $\bo\notin P$, and let $B$ be the unique half-space bounded by $P$ that contains $\bo$.  For any $\bx\in P$, the ray from $\bo$ through $\bx$ thus intersects $B$ in the segment $\{t\bx\,|\,0\leq t\leq 1\}$.  But for $\bx\in U_{n+1}$, $r_n(\bx) = \bx/\sqrt{-\bx\circ\bx}$ has smaller $t$-coordinate than $1$ by definition.  Finally, recall from Lemma \ref{convex component} that any half-space $B$ with $\bo\in B$ has $B\cap\mathbb{H}^n$ convex.\end{proof}

\begin{lemma}\label{convex hull boundary}Let $\cals\subset\mathbb{H}^n$, and take $r_n\co U_{n+1}\to\mathbb{H}^n$ as in Lemma \ref{convex hull}.  For any $\bx\in\mathbb{H}^n$ that lies outside  $r_n(\mathrm{Hull}(\cals))$ but in its closure, there is an $n$-dimensional time-like subspace $V$ through $\bx$ such that $\cals$ is contained in one of the half-spaces bounded by $V$.\end{lemma}

\begin{proof} Let $\{\bx_k\}$ be a sequence of visible points of $\mathrm{Hull}(\cals)$ such that $r_n(\bx_k) \to\bx$ as $k\to\infty$.  If $\bx_k\circ\bx_k$ were universally bounded then since $r_n(\bx_k)= \bx_k/\sqrt{-\bx_k\circ\bx_k}$ a subsequence would converge to some $\bx_0\in\mathrm{Hull}(\cals)$ mapping to $\bx$, so we will assume $\bx_k\circ\bx_k\to-\infty$ as $k\to\infty$.

For each $k$ let $P_k$ be a support plane for $\mathrm{Hull}(\cals)$ through $\bx_k$ supplied by Lemma \ref{visible point}, separating $\bo$ from $\mathrm{Hull}(\cals)$.  Let $\eta_k$ be a unit Euclidean normal to $P_k$ pointing outward from $\mathrm{Hull}(\cals)$, so by construction $\eta_k\cdot(\bo-\bx_k)\geq 0$ and $\eta_k\cdot(\by-\bx_k)\leq 0$ for any $k\in\mathbb{N}$ and $\by\in\mathrm{Hull}(\cals)$.

Passing to a subsequence, we assume that $\{\eta_k\}$ converges to a unit vector $\eta$.  Fixing $\by\in\mathrm{Hull}(\cals)$, dividing each of the above inequalities by $\sqrt{-\bx_k\circ\bx_k}$, and taking a limit shows that $\eta\cdot(\bo-\bx) = 0$, so the affine $n$-plane $V$ through $\bx$ and normal to $\eta$ also contains $\bo$ (hence is a subspace).  We also note for any $k\in\mathbb{N}$ that
$$ \eta_k\cdot\left(\by-\frac{\bx_k}{\sqrt{-\bx_k\circ\bx_k}}\right) = \eta_k\cdot (\by-\bx_k) + \left(1-\frac{1}{\sqrt{-\bx_k\circ\bx_k}}\right)\eta_k\cdot\bx_k \leq 0. $$
Since $\by\in\mathrm{Hull}(\cals)$ was arbitrary, the translate of $P_k$ through $r_n(\bx_k)$ separates $\mathrm{Hull}(\cals)$ from $\bo$.  Since $V$ is the limit (in the sense of Lemma \ref{plane convergence}) of these translates, it does as well. \end{proof}

\section{The Delaunay tessellation and the finite case}\label{Delaunay existence}  

\begin{definition}\label{faces n stuff}  The \textit{dimension} of a set $C\subset\mathbb{R}^{n+1}$ is the minimal dimension of an affine plane $Q$ containing $C$, and the \textit{interior} of $C$ is its interior in $Q$ (cf.~\cite[Prop.~11.2.7]{Berger}).  A \textit{face} of a closed convex set $C$ is a set of the form $C\cap P$, where $P$ is a support plane for $C$.  If $\bo\notin C$, a face $F$ of $C$ is \textit{visible} if every $\bx\in F$ is visible.

\begin{figure}
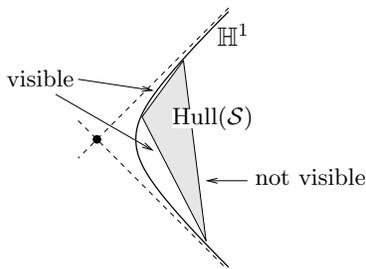
\caption{The convex hull in $\mathbb{R}^2$ of a three-point set $\cals$ in $\mathbb{H}^1$}
\label{in R^2}
\end{figure}

We define the \textit{Delaunay tessellation} of a set $\cals\subset\mathbb{H}^n= \{\bx\in\mathbb{R}^{n+1}\,|\,\bx\circ\bx=-1, x_0>0\}$ as: 
$$\{r_n(F)\,|\, F\ \mbox{is a visible face of}\ \mathrm{Hull}(\cals)\}.$$
Here $r_n$ is as in Lemma \ref{convex hull}.  For a visible face $F$ of $\mathrm{Hull}(\cals)$, say $r_n(F)$ is a \textit{Delaunay cell}.\end{definition}

The main result of this section, Proposition \ref{the finite case}, asserts among other things that when $\cals$ is finite its Delaunay tessellation is a polyhedral complex in the standard sense (see eg.~\cite[Dfn.~2.1.5]{DeLoRS}): cells are polyhedra, their faces are also cells, and cells that intersect do so in a face.  Before going further let us port these notions to the hyperbolic setting.

\begin{definition}\label{tessellation}  A \textit{convex polyhedron} in $\mathbb{H}^n$ is the intersection of a collection of hyperbolic half-spaces (defined in Lemma \ref{convex component}) whose associated collection of bounding totally geodesic subspaces is locally finite.  The \textit{dimension in} $\mathbb{H}^n$ of a set $C$ is defined, in analogy with \ref{faces n stuff}, to be the minimal $k$ such that $C$ is contained in a $k$-dimensional totally geodesic subspace.  The \textit{interior} and \textit{faces} of a convex set are similarly defined in analogy with \ref{faces n stuff}.\end{definition}

The key advantage of taking $\cals$ finite is that it ensures $\mathrm{Hull}(\cals)$ is a \textit{polyhedron} of $\mathbb{R}^{n+1}$ \cite[Prop.~12.1.15]{BergerII}, the intersection of a finite collection of half-spaces (cf.~\cite[Definition 12.1.1]{BergerII}).

\begin{lemma}\label{poly to poly}  If $\cals\subset\mathbb{H}^n$ is finite and of dimension $k\leq n$ in $\mathbb{R}^{n+1}$, and the $k$-dimensional affine subspace containing $\cals$ does not also contain $\bo$, then $r_n$ takes $\mathrm{Hull}(\cals)$ to the closed convex hull of $\cals$ in $\mathbb{H}^n$.  This is a compact polyhedron of dimension $k$ in $\mathbb{H}^n$.\end{lemma}

\begin{proof}  Take $k = n$ and let $P$ be the affine subspace of $\mathbb{R}^{n+1}$ containing $\cals$ but not $\bo$.  By \cite[Prop.~12.1.15]{BergerII}, $\mathrm{Hull}(\cals)$ is a compact, convex polyhedron in $P$, so $\mathrm{Hull}(\cals) = \bigcap_{i=1}^j H_i$, where the $H_i$ are half-spaces of $P$ bounded by $(n-1)$-dimensional affine subspaces $Q_i\subset P$.  Any point in the frontier of $P$ is in some $Q_i$.

For each $i$ let $V_i$ be the vector subspace of $\mathbb{R}^{n+1}$ containing $Q_i$, and let $B_i$ be the half-space bounded by $V_i$ that contains $H_i$.  Since $r_n$ preserves $V_i$ and the $B_i$ for each $i$, $r_n(\mathrm{Hull}(\cals)) \subset \bigcap_{i=1}^j (B_i\cap\mathbb{H}^n)$ is a compact subset with its frontier in the $V_i\cap\mathbb{H}^n$.  It follows that equality holds.  By compactness and Lemma \ref{convex hull}, $r_n(\mathrm{Hull}(\cals))$ is the closed convex hull of $\cals$ in $\mathbb{H}^n$.

If $r_n(\mathrm{Hull}(\cals))$ were contained in $V\cap\mathbb{H}^n$ for some proper vector subspace $V$ then $\cals$ would lie in $V\cap P$, which has dimension $n-1$.  Therefore $r_n(\mathrm{Hull}(\cals))$ has dimension $n$.

If $\cals$ has dimension $k<n$ there is a $(k+1)$-dimensional time-like vector subspace $V$ containing $\cals$, whose intersection with $\mathbb{H}^n$ is an isometrically embedded copy of $\mathbb{H}^k$.  We therefore work in $V$ and apply the full-dimensional case.\end{proof}

\begin{lemma}\label{good visibility}  If $\cals\subset\mathbb{H}^n$ is finite then every visible point of $\mathrm{Hull}(\cals)$ is contained in a visible face, and every visible face is of the form $P\cap\mathrm{Hull}(\cals)$ for a support plane $P$ with $\bo\notin P$ that separates $\bo$ from $\cals$.  Each face of a visible face of $\mathrm{Hull}(\cals)$ is itself a visible face of $\mathrm{Hull}(\cals)$.\end{lemma}

\begin{proof}  
Applying \cite[Prop.~12.1.15]{BergerII}, let $\mathrm{Hull}(\cals)=\bigcap_{i=1}^k B_i$ for a collection of half-spaces $B_i$ bounded by affine $n$-planes $P_i$.  If $\bx\in\mathrm{Hull}(\cals)$ then $\bx\in B_i$ for all $i$, and if $\bx$ is visible then for each $t\in [0,1)$ there exists $i_t$ such that $t\bx\notin B_{i_t}$.  Since the collection $\{B_i\}$ is finite, as $t\to 1$ a subsequence of the $i_t$ is constant at some $i_0$.  Since $\bo\notin B_{i_0}$, $F_{i_0}$ is visible.

Now suppose $\bx$ is in the interior of a visible face $F$.  Let $P$ be a support plane for $\mathrm{Hull}(\cals)$ with $F=P\cap\mathrm{Hull}(\cals)$, and let $\eta$ be a Euclidean normal to $P$ pointing outward from $\mathrm{Hull}(\cals)$.  If $\eta_{i_0}$ is a Euclidean normal to $P_{i_0}$ pointing outward from $\mathrm{Hull}(\cals)$, then $\eta_t = t\eta+(1-t)\eta_{i_0}$ is normal to a support plane $P_t$ for $\mathrm{Hull}(\cals)$ for each $t\in[0,1]$, pointing outward from $\mathrm{Hull}(\cals)$.  For $t>0$, $F = P_t\cap\mathrm{Hull}(\cals)$, and for $t$ near $0$, $\eta_t\cdot\bx<0$ since this is true for $\eta_{i_0}$.  For such $t$, $\bo\notin P_t$ and $P_t$ separates it from $\cals$.

For a visible face $F$ of $\mathrm{Hull}(\cals)$, let $P$ be a support plane for $\mathrm{Hull}(\cals)$ separating $\bo$ from $\cals$ with $\bo\notin\cals$.  For any proper face $F_0\subset F$ there is a codimension-one support plane $Q$ for $F$ in $\cals$ with $F_0=Q\cap F$.  Let $\eta$ be a Euclidean normal to $P$ pointing outward from $\cals$ and $\delta$ a Euclidean normal to $Q$ in $P$ pointing outward from $F$.  One verifies that for small $t$, $\eta_t = (1-t)\eta + t\delta$ is a Euclidean normal to a support plane $P_t$ in $\mathbb{R}^{n+1}$ for $\mathrm{Hull}(\cals)$ with $F_0 = P_t\cap\mathrm{Hull}(\cals)$ that separates $\cals$ from $\bo$.\end{proof}

\begin{proposition}\label{the finite case}  The Delaunay tessellation of a finite set $\cals\subset\mathbb{H}^{n}$ is a decomposition of the closed convex hull of $\cals$ in $\mathbb{H}^n$ into a finite collection of convex polyhedra (the \mbox{\rm cells}), such that each face of each cell is a cell and distinct cells that intersect do so in a face of each.  For each non-totally geodesic hypersphere $S$ of $\mathbb{H}^n$ that intersects $\cals$ and bounds a convex region $B$ with $B\cap\cals = S\cap\cals$, the closed convex hull of $S\cap\cals$ in $\mathbb{H}^n$ is a Delaunay cell.  Each Delaunay cell is compact and of this form.\end{proposition}

\begin{remark}\label{empty characterization}  The \textit{hyperbolic Delaunay complex} of \cite{BDT}, defined on p.~7 there, satisfies the empty circumspheres condition with only metric spheres.  Thus it is a subcomplex of our Delaunay tessellation (by Theorem \ref{the real geometric dual}, the geometric dual to the Voronoi tessellation). \end{remark}

\begin{proof}[Proof of Prop.~\ref{the finite case}]  Since $\cals$ is finite, $\mathrm{Hull}(\cals)$ is a compact convex polyhedron \cite[Prop.~12.1.15]{BergerII}.  By compactness $r_n(\mathrm{Hull}(\cals))$ is closed.  This coincides with the image of its visible set, so by Lemma \ref{convex hull} the visible set maps onto the closed convex hull of $\cals$ in $\mathbb{H}^n$.  By Lemma \ref{good visibility}, each visible point of $\mathrm{Hull}(\cals)$ is in an $n$-dimensional visible face, so the Delaunay tessellation covers the closed  convex hull of $\cals$.  By Lemma \ref{poly to poly} it is the union of its $n$-cells.

For a Delaunay cell $C$ with a face $C_0\subset C$, let $V_0$ be the time-like vector subspace such that $V_0\cap\mathbb{H}^n$ is the support plane for $C$ with $C_0 = V_0\cap C$.  For the visible face $F$ of $\mathrm{Hull}(\cals)$ such that $C=r_n(F)$, and the support plane $P$ for $\mathrm{Hull}(\cals)$ supplied by Lemma \ref{good visibility}, $V_0\cap P$ is a codimension-one support plane for $F$ in $P$ such that $r_n(V_0\cap F) = C_0$.  The face $F_0 = V_0\cap F$ of $F$ is also a face of $\mathrm{Hull}(\cals)$, hence $C_0$ is also a Delaunay cell.

Now let $C$ and $C'$ be distinct, intersecting Delaunay cells, let $F$ and $F'$ be the visible faces with $C=r_n(F)$ and $C'=r_n(F')$, and let $P$ and $P'$ be support planes for $\mathrm{Hull}(\cals)$ as in Lemma \ref{good visibility}, with $F = P\cap\mathrm{Hull}(\cals)$ and $F'=P'\cap\mathrm{Hull}(\cals)$.  Then $Q = P\cap P'$ is a support plane for $F$ in $P$ with $F\cap F' = F\cap Q$, and the necessarily time-like vector subspace $V$ spanned by $Q$ intersects $\mathbb{H}^n$ in a support plane for $C$ with $C \cap Q = r_n(F\cap Q) = r_n(F\cap F') = C\cap C'$.  (Recall that $r_n$ is injective on the visible set of $\mathrm{Hull}(\cals)$, hence on $F\cup F'$.) It follows that $C\cap C'$ is a face of $C$.  By the same argument it is a face of $C'$.

For a visible face $F$ let $P$ be the support plane supplied by Lemma \ref{good visibility}, with $F = P\cap\mathrm{Hull}(\cals)$ and $\bo$ and $\cals$ in opposite half-spaces bounded by $P$.  Since $\bo\notin P$, $S=P\cap\mathbb{H}^n$ is not totally geodesic.  By Lemma \ref{convex component} the half-space $B$ containing $\bo$ has convex intersection with $\mathbb{H}^n$, and $r_n(F)\subset B\cap\mathbb{H}^n$ by Lemma \ref{ratatouille}.  By Lemma \ref{poly to poly}, $r_n(F)$ is a compact polyhedron, the closed convex hull of $B\cap\cals = S\cap\cals$.

Now suppose on the other hand that $\cals_0 = \cals\cap S$ for some hypersphere $S$ of the form $P\cap\mathbb{H}^n$, where $P$ is an affine plane bounding a half-space $B$ with $B\cap\mathbb{H}^n$ convex and $B\cap\cals = \cals_0$.  Then $P$ is a support plane for $\mathrm{Hull}(\cals)$, and $F = P\cap \mathrm{Hull}(\cals)$ is a face with $F = \mathrm{Hull}(\cals_0)$.  Since $S$ is not totally geodesic $P$ does not contain $\bo$, so $F$ is visible.\end{proof}

\section{A bad example}\label{bad example}

The main example of this section, a set $\cals\subset\mathbb{H}^2$, illustrates the issues one has to work around in analyzing the Delaunay tessellation of an arbitrary locally finite subset of $\mathbb{H}^n$.  Note that if $\cals$ is locally finite in $\mathbb{H}^n$, it is also locally finite in $\mathbb{R}^{n+1}$: for $\bx$ and $\by\in\mathbb{H}^n$,\begin{align}\label{x - y}
  \|\bx-\by\|^2 = 2\left[(x_0-y_0)^2 - \bx\circ\by - 1\right] \geq 2(-\bx\circ\by-1) \end{align}
(Here the left-hand side is the square of the Euclidean distance between $\bx$ and $\by$).  Such $\cals$ may however experience ``convergence at infinity'', causing pathologies in the Delaunay tessellation.  What we mean by this is most apparent in the upper half-plane model, where we introduce our bad example $\cals$ before translating back to the hyperboloid model.


\begin{claim}  There exists $p_0$ on the positive imaginary axis and a sequence $\{p_n = x_n+iy_n\}\subset\mathbb{H}^2$ such that $x_n$ decreases to some $a>0$, $y_n$ decreases to $0$, and for each $n\geq 2$, $p_n$ is the lowest point of the Euclidean circle through $p_0$, $p_{n-1}$, and $p_n$.\end{claim}

Let us take the Claim for granted for now (we will prove it at the end of the section) and record its consequences for Delaunay tessellations.  For $p_0$ and $\{p_n\}$ provided by the Claim, let $p_{-n} = -\bar{p}_n$ for each $n\in\mathbb{N}$, and let $\cals = \{p_n\,|\,n\in\mathbb{Z}\}$.  By construction,  $p_{\infty} \doteq a\in\mathbb{R}$ and $p_{-\infty} \doteq -a$ are the only accumulation points of $\cals$ in $\mathbb{R}^2$.  Since these lie outside $\mathbb{H}^2$, $\cals$ is closed and discrete, and hence locally finite, in $\mathbb{H}^2$.

The Claim directly implies for each $n\geq 2$ that the Euclidean circle $S_n$ through $p_0$, $p_{n-1}$ and $p_n$ lies entirely in $\mathbb{H}^2$, so it is the hyperbolic circle through these points.  Moreover, convexity properties of circles imply that the closed disk bounded by $S_n$ contains only $p_0$, $p_{n-1}$ and $p_n$ among all $p_k\in\cals$.  The analogous fact holds for the circle $S_{-n}$ through $p_0$, $p_{-n}$ and $p_{-n+1}$.

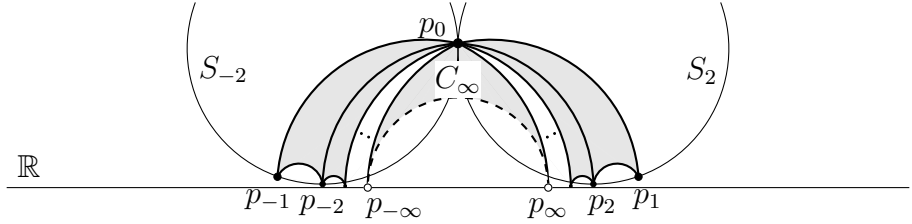
\begin{figure}
\begin{tikzpicture}

\begin{scope}[scale=1.2]

\fill [opacity=0.1] (2,0.12) arc (4.194:102.803:1.64);
\fill [opacity=0.1] (2,0.12) -- (0,1.6) -- (1.5,0.039);
\draw [thick] (2,0.12) arc (4.194:102.803:1.64);
\fill [opacity=0.1] (1.5,0.039) -- (0,1.6) -- (1.25,0.017);
\draw [thick] (1.5,0.039) arc (1.394:86.317:1.603);
\fill [color=white] (1.25,0.017) arc (0.59:76:1.65);
\draw [thick] (1.25,0.017) arc (0.59:76:1.65);
\fill [color=white] (2,0.12) arc (26.84:171.56:0.266);
\draw [thick] (2,0.12) arc (26.84:171.56:0.266);
\fill [color=white] (1.5,0.039) arc (17.65:172.4:0.1285);
\draw [thick] (1.5,0.039) arc (17.65:172.4:0.1285);

\fill [opacity=0.1] (-2,0.12) arc (175.806:77.197:1.64);
\fill [opacity=0.1] (-2,0.12) -- (0,1.6) -- (-1.5,0.039);
\draw [thick] (-2,0.12) arc (175.806:77.197:1.64);
\fill [opacity=0.1] (-1.5,0.039) -- (0,1.6) -- (-1.25,0.017);
\draw [thick] (-1.5,0.039) arc (178.606:93.683:1.603);
\fill [color=white] (-1.25,0.017) arc (179.41:104:1.65);
\draw [thick] (-1.25,0.017) arc (179.41:104:1.65);
\fill [color=white] (-2,0.12) arc (153.16:8.44:0.266);
\draw [thick] (-2,0.12) arc (153.16:8.44:0.266);
\fill [color=white] (-1.5,0.039) arc (162.35:7.6:0.1285);
\draw [thick] (-1.5,0.039) arc (162.35:7.6:0.1285);

\fill [opacity=0.1] (1,0) arc (0:64.01:1.78);
\fill [opacity=0.1] (1,0) -- (0,1.6) -- (-1,0);
\draw [thick] (1,0) arc (0:64.01:1.78);
\fill [opacity=0.1] (-1,0) arc (180:115.99:1.78);
\draw [thick] (-1,0) arc (180:115.99:1.78);
\fill [color=white] (1,0) arc (0:180:1);
\draw [thick, dashed] (1,0) arc (0:180:1);

\draw (-5,0) -- (5,0);
\node [above right] at (-5,0) {$\mathbb{R}$};

\draw [fill] (0,1.6) circle [radius=0.05];
\node [left] at (0,1.77) {$p_0$};

\draw [fill] (-2,0.12) circle [radius=0.04];
\node [below] at (-2.1,0.1) {$p_{-1}$};
\draw (0,1.6) arc (2.304:-200:1.501);
\draw (0,1.6) arc (2.304:20:1.501);
\node at (-2.6,1.3) {$S_{-2}$};

\draw [fill] (2,0.12) circle [radius=0.04];
\node [below] at (2.1,0.1) {$p_1$};
\draw (0,1.6) arc (-182.304:20:1.501);
\draw (0,1.6) arc (-182.304:-200:1.501);
\node at (2.7,1.3) {$S_2$};


\draw [fill] (-1.5,0.039) circle [radius=0.03];
\node [below] at (-1.5,0.039) {$p_{-2}$};
\draw [fill] (1.5,0.039) circle [radius=0.03];
\node [below] at (1.6,0.039) {$p_{2}$};

\fill [color=white] (0.25,1) -- (0.25,1.4) -- (-0.25,1.4) -- (-0.25,1);
\node at (0,1.2) {$C_{\infty}$};

\draw [fill] (1.25,0.017) circle [radius=0.015];
\draw [fill] (-1.25,0.017) circle [radius=0.015];

\draw [fill] (1.07,0.6) circle [radius=0.01];
\draw [fill] (0.99,0.56) circle [radius=0.01];
\draw [fill] (-1.07,0.6) circle [radius=0.01];
\draw [fill] (-0.99,0.56) circle [radius=0.01];

\draw [fill, color=white] (1,0) circle [radius=0.04];
\draw (1,0) circle [radius=0.04];
\node [below] at (1,0) {$p_{\infty}$};
\draw [fill, color=white] (-1,0) circle [radius=0.04];
\draw (-1,0) circle [radius=0.04];
\node [below] at (-0.7,0) {$p_{-\infty}$};
\end{scope}

\end{tikzpicture}
\caption{The example constructed in this section.  Delaunay cells are shaded.}
\end{figure}

We now translate $\cals$ to the hyperboloid model using an isometry $I$.   For each $n\geq 2$, since $S_n$ and $S_{-n}$ are metric circles in $\mathbb{H}^2$, their images under $I$ are the intersections of affine planes $P_{\pm n}$ with $\mathbb{H}^2$ by Lemma \ref{classification of spheres}.  Since the disk bounded by $S_n$ intersects only $p_0$, $p_n$ and $p_{n-1}$, by Lemma \ref{convex component} the same holds for the half-space $B_n$ containing $\mathbf{0}$ with $\partial B_n = P_n$.  It follows that $P_n\cap\mathrm{Hull}(I(\cals))$ is a visible face of $\mathrm{Hull}(I(\cals))$, hence that $\cals$ has a Delaunay cell $C_n$ that is a triangle with vertices $p_0$, $p_{n-1}$, and $p_n$.  Another, $C_{-n}$, is spanned by $p_0$, $p_{-n+1}$ and $p_n$.

We claim that the Delaunay tessellation of $\cals$ has another cell $C_{\infty}$ that is not of the form $C_{\pm n}$ for any $n\geq 2$, which occupies all of the triangle spanned by $p_0$ and the ideal points $p_{\pm\infty}$ of $\mathbb{H}^2$ except the geodesic joining $p_{-\infty}$ to $p_{\infty}$.  The Euclidean circle $S_{\infty}$ through these three points does not enclose any other points of $\cals$.  This can be argued using the facts that the $S_{\pm n}$ limit to circles through $p_0$ that are tangent to $\mathbb{R}$ at $p_{\pm\infty}$, respectively, as $n\to\infty$; that these circles each exclude all points of $\cals-\{p_0\}$; and that the portion of $S_{\infty}$ consisting of points with real coordinate of absolute value at least $a$ is enclosed by these circles.

The image of $S_{\infty}\cap\mathbb{H}^2$ under $I$ is $P_{\infty}\cap\mathbb{H}^2$, where $P_{\infty}$ is an affine plane parallel to a time-like vector subspace spanned by the two lines $\ell_{\pm\infty}$ in the light cone corresponding to the ideal points $p_{\pm\infty}$ of the upper half-plane model.  This is a support plane for $\mathrm{Hull}(I(\cals))$, since it contains $I(p_0)$ and separates all other points of $\cals$ from $\mathbf{0}$ (by Lemma \ref{convex component} and the paragraph above).  We take $C_{\infty} = r_3(F_{\infty})$, where $F_{\infty}\doteq P_{\infty}\cap\mathrm{Hull}(\cals)$.  The description of $C_{\infty}$ above will follow from the fact that $F_{\infty}$ is the wedge spanned by the rays through $I(p_0)$ that are parallel to $\ell_{\pm\infty}$ and contained in $U_3$.

This in turn follows from the fact that the lines through the origin and the $I(p_{\pm n})$ converge in projective space to $\ell_{\pm\infty}$ as $n\to\infty$, which is the analog for the hyperboloid model of the convergence $p_{\pm n}\to p_{\pm\infty}$ in $\mathbb{R}^2$.  The fact implies that the line segments in $\mathbb{R}^3$ from $I(p_0)$ to the $I(p_{\pm n})$ have direction vectors converging to those of $\ell_{\pm\infty}$.  The line segments themselves have Euclidean lengths increasing without bound, since $\cals$ is locally finite but infinite, so each point on one of the rays in question is approached by a sequence on these line segments.  Thus since $\mathrm{Hull}(\cals)$ is closed it contains the rays, and the claim follows.

The Delaunay tessellation of $\cals$ lacks several properties that hold for finite sets by Proposition \ref{the finite case}, and for lattice-invariant sets as described in Theorem \ref{pseudo EP}.\begin{itemize}
  \item  The Delaunay cell $C_{\infty}$ is not the convex hull of its vertices, since it has only one: $p_0$.  Neither it nor the union of all Delaunay cells is a closed subset of $\mathbb{H}^2$, since they are missing the geodesic joining $p_{-\infty}$ and $p_{\infty}$.
  \item  The faces of $F_{\infty}$ containing $I(p_0)$ are not faces of $\mathrm{Hull}(\cals)$, since the only support plane that contains them is $P_{\infty}$ (c.f.~\cite[Remark 2.12(ii)]{AkSak}).  Consequentially, the two edges of $C_{\infty}$ are not themselves Delaunay cells.
  \item  The collection of Delaunay cells of $\cals$ is not locally finite at $p_0$ or any other point of an edge of $C_{\infty}$.\end{itemize}

In a slight variation on this construction, one could take $a=0$ in the Claim above.  Then $\cals\subset\mathbb{H}^2$ would have the single accumulation point $p_{\infty} =0\in\mathbb{R}$.  In this case the corresponding face $F_{\infty}$ of $\mathrm{Hull}(\cals)$ is one-dimensional --- the geodesic arc joining $p_0$ to $p_{\infty}$ --- and contained in no $2$-face.


\begin{proof}[Proof of claim]  Take $p_0=i$ for starters.  For $n\in\mathbb{N}$ we will take $p_n = x_n+iy_n$, where for each $n\in\mathbb{N}$, $x_n = a + r^n$ for some $a>0$ and $r\in(0,1)$ and for each $n\geq 2$, $p_n$ be the lowest point of a circle containing $p_0$ and $p_{n-1}$.  This requirement yields a quadratic equation whose smaller solution gives the following formula for $y_n$:\begin{align*}
  y_n & = \frac{1+2x_{n-1}x_n - x_n^2 - y_{n-1}^2 - \sqrt{(x_n^2+(1-y_{n-1})^2)((2x_{n-1}-x_n)^2+(1-y_{n-1})^2)}}{2(1-y_{n-1})} \end{align*}
This formula determines $y_n$ in terms of $x_n$, $x_{n-1}$, and $y_{n-1}$ for $n\geq 2$.  It is clear by construction that the sequence $\{y_n\}$ is strictly decreasing, since $p_n$ and $p_{n-1}$ both lie on $S_n$ but $p_n$ is its unique lowest point.  We will show that if $y_1$ is chosen in the interval $(0,1)$ then $\{y_n\}$ is also bounded below, so it converges to some $L<y_1$.  Translating the entire original collection $\{p_0,p_1,\hdots\}$ vertically by $-L$ then yields the desired one.

Note that since the sequence $\{x_n\}$ is decreasing, $x_n$ is strictly less than $2x_{n-1} - x_n$, so:\begin{align*}
  y_n & > \frac{1+2x_{n-1}x_n - x_n^2 - y_{n-1}^2 - \left[(2x_{n-1}-x_n)^2+(1-y_{n-1})^2\right]}{2(1-y_{n-1})} \\
	& = y_{n-1} - \frac{(x_{n-1}-x_n)(2x_{n-1}-x_n)}{1-y_{n-1}} > y_{n-1} - K r^{n-1}, \end{align*}
where $K = \frac{(1-r)(a+2)}{1-y_1}$.  This bounds $\{y_n\}$ below by a convergent geometric sequence, and the claim follows.\end{proof}

\section{The Voronoi tessellation and its geometric dual}\label{into Voronoi}

Here we will introduce the Voronoi tessellation of a locally finite set $\cals\subset\mathbb{H}^n$, then describe its ``geometric dual'', a subcomplex of the Delaunay tessellation of $\cals$.  For any such $\cals$ the Voronoi tessellation is a locally finite \textit{polyhedral complex}, in the sense of \cite[Dfn.~2.1.5]{DeLoRS}, and its geometric dual is a polyhedral complex (possibly not locally finite, see Remark \ref{bad dual} below).

\begin{definition}\label{Voronoi}  For locally finite $\cals\subset \mathbb{H}^n$ and $\bs\in\cals$, the \textit{Voronoi $n$-cell} determined by $\bs$ is
$$ V_{\bs} = \{\bx\in\mathbb{H}^n\,|\, d_H(\bs,\bx)\leq d_H(\bs',\bx)\ \forall\ \bs'\in\cals\}$$
The \textit{Voronoi tessellation} of $\cals$ is the complex with cells consisting of the $V_{\bs}$ and their faces.\end{definition}

It is clear from the definition that $\mathbb{H}^n$ is the underlying space of the Voronoi tessellation.

\begin{lemma}  If $\cals\subset\mathbb{H}^n$ is locally finite then $V_{\bs}$ is a convex polyhedron for each $\bs\in\cals$, as is each face of $V_{\bs}$, and the collection $\{V_{\bs}\,|\,\bs\in\cals\}$ is locally finite.  For any $\{\bs_0,\hdots,\bs_l\}$ such that $\bigcap_{i=0}^l V_{\bs_i}$ is non-empty, it is a face of $V_{\bs_i}$ for each $i$.\end{lemma}

Using the definition of the hyperbolic metric (cf.~Section \ref{intro Lorentzian}), we may recast the criterion $d_H(\bx,\bs) \leq d_H(\bx,\bs')$ as $\bx\circ\bs \geq \bx\circ\bs'$.  In particular, the locus where equality holds is the hyperplane $\{\bx\circ(\bs-\bs') = 0\}$, which is time-like since $\bs-\bs'$ is space-like by (\ref{x circ y}).  So the set $\{\bx\,|\,d_H(\bx,\bs) \leq d_H(\bx,\bs')\}$ is a hyperbolic half-space.  More generally:

\begin{fact}\label{equations n stuff}  For $\{\bs_0,\hdots,\bs_{l}\}\subset\mathbb{H}^n$, if the locus $\{\bx\,|\,d_H(\bx,\bs_i) = d_H(\bx,\bs_j)\ \forall\ i,j\}$ is non-empty then it is an $(n-k)$-dimensional totally geodesic subspace of $\mathbb{H}^n$, where $k$ is the dimension of the minimal such subspace containing all $\bs_i$.\end{fact}

\begin{proof} As above we restate the criterion that $d_H(\bx,\bs_i) = d_H(\bx,\bs_j)$ for all $i$ and $j$ as $\bx\circ(\bs_0-\bs_i)=0$ for all $i>0$.  We may assume that $\bs_0,\hdots,\bs_k$ is linearly independent and spans the span $V$ of $\{\bs_0,\hdots,\bs_l\}$ in $\mathbb{R}^{n+1}$.  Another linearly independent spanning set for $V$ is $\{\bs_0,\bs_0-\bs_1,\hdots,\bs_0-\bs_k\}$, and the locus in question lies in the Lorentz-orthogonal complement to the span $V_0$ of the last $k$.  This has dimension $(n+1-k)$, so if every $\bx\in V_0^{\perp}$ is also orthogonal to $\bs_0-\bs_j$ for all $j>k$ then the fact will follow.  

For any such $j$ there exist $a_i\in\mathbb{R}$ such that $\bs_j = \sum_{i=0}^k a_i\bs_i$, or equivalently,
$$ \bs_0-\bs_j = \left(1-\sum_{i=0}^k a_i\right)\bs_0 + \sum_{i=1}^k a_i(\bs_0-\bs_i). $$
The coefficient of $\bs_0$ above is non-zero if and only if $\bs_0-\bs_j$ is not in $V_0$.  If this holds then any $\bx\in V_0^{\perp}$ that is also orthogonal to $\bs_0-\bs_j$ must be orthogonal to $\bs_0$.  But since $\bs_0$ is time-like, $\bs_0^{\perp}$ is space-like, so the locus $\{\bx\,|\,d_H(\bx,\bs_i) = d_H(\bx,\bs_j)\ \forall\ i,j\}$ is empty in this case.\end{proof}

Let us now prove the lemma.

\begin{proof}  For any $\bx\in\mathbb{H}^n$, by local finiteness the function $\bs\mapsto d_H(\bs,\bx)$ attains a minimum $J$ at some $\bs_0\in\cals$; hence $\bx\in V_{\bs_0}$.  The closed ball $B_{\bx}(2J)$ contains only finitely many points $\{\bs_0,\hdots,\bs_l\}$ of $\cals$, again by local finiteness.  Then $B_{\bx}(J/2)\subset \bigcup_{i=0}^l V_{\bs_i}$, and $B_{\bx}(J/2)\cap V_{\bs_0} \subset \bigcap_{i=1}^l \calh_i$, where $\calh_i = \{\by\,|\,d_H(\by,\bs_0)\leq d_H(\by,\bs_i)\}$ for each $i$ between $1$ and $l$.

This implies local finiteness both for the collection $\{V_{\bs}\}$ and for the collection of bounding hyperplanes of each $V_{\bs}$.  Thus $V_{\bs}$ is a polyhedron (recall Definition \ref{tessellation}), as are its faces.

Suppose $\bigcap_{i=0}^l V_{\bs_i}$ is non-empty.  Note that the vectors $\bu_i\doteq\bs_0-\bs_i$ featured in the proof of Fact \ref{equations n stuff} are Lorentz-normal to the vector subspaces $V_i$ with the property that $V_i\cap\mathbb{H}^n$ is a bounding hyperplane for $\calh_i$.  Further, $\by\circ\bu_i \geq 0$ for each $\by\in\calh_i$.  Any positive linear combination of the $\bu_i$ thus determines a support plane for $V_{\bs_0}$ intersecting it in $\bigcap_{i=0}^l V_{\bs_i}$.\end{proof}

\begin{proposition}\label{geometric dual}  Let $\cals\subset\mathbb{H}^n$ be locally finite.  For a Voronoi $k$-cell $V$, if $\cals_0\subset\cals$ is maximal such that $V=\bigcap_{\bs\in\cals_0} V_{\bs}$ then $F_V = \mathrm{Hull}(\cals_0)$ is an $(n-k)$-dimensional face of $\mathrm{Hull}(\cals)$,  of the form $P\cap\mathrm{Hull}(\cals)$ for a support plane $P$ for $\mathrm{Hull}(\cals)$ that separates $\cals$ from $\bo$ and is parallel to a space-like subspace of $\mathbb{R}^{n+1}$.  For any such support plane $P$, if $\cals_0= P\cap\cals$ then $P\cap\mathrm{Hull}(\cals) = \mathrm{Hull}(\cals_0)$ and $\bigcap_{\bs\in\cals_0} V_{\bs}$ is a Voronoi cell.\end{proposition}

\begin{corollary}\label{black sox}  For a Voronoi $k$-cell $V$ of a locally finite set $\cals\subset\mathbb{H}^n$, if $\cals_0\subset\cals$ is maximal such that $V = \bigcap_{\bs\in\cals_0} V_{\bs}$ then the closed convex hull $C_V$ of $\cals_0$ in $\mathbb{H}^n$, the \mbox{\rm geometric dual} to $V$, is a Delaunay cell and an $(n-k)$-dimensional, compact, convex polyhedron.\end{corollary}

If $\bx$ is in a Voronoi cell $V= \bigcap_{\bs\in\cals_0} V_{\bs}$ then by definition all $\bs\in\cals_0$ are equidistant from $V$.  Since metric spheres are compact (Lemma \ref{convex component}), any such collection is finite and Corollary \ref{black sox} follows directly from Proposition \ref{geometric dual} and Lemma \ref{poly to poly}, with $C_V = r_n(F_V)$.

One step in the proof of Proposition \ref{geometric dual} uses an argument borrowed from \cite{EpstePen}.

\begin{lemma}\label{rotate space}  For locally finite $\cals\subset\mathbb{H}^n$, suppose $P$ is a support plane for $\mathrm{Hull}(\cals)$, parallel to a space-like subspace of $\mathbb{R}^{n+1}$, that separates $\cals$ from $\bo$.  Then $F = P\cap\mathrm{Hull}(\cals)$ is a compact polyhedron equal to $\mathrm{Hull}(\cals_0)$, where $\cals_0=P\cap\cals$, and every face of $F$ is a face of $\mathrm{Hull}(\cals)$.\end{lemma}

\begin{proof}  $P\cap U_{n+1}$ is compact by Lemma \ref{convex component}, so therefore $F$ is as well.  Since $F$ contains $\cals_0$ it contains $\mathrm{Hull}(\cals_0)$; if properly then for any $\bx\in F - \mathrm{Hull}(\cals_0)$ there is a support plane $Q$ for $\mathrm{Hull}(\cals_0)$ in $P$ that separates it from $\bx$.  Rotating $P$ around $Q$ by a small amount produces a support plane $P'$ for $\mathrm{Hull}(\cals)$ that separates it from $\bx$ and $\bo$, contradicting that $\bx\in F$.

To flesh this out a bit, if $\eta$ is a Euclidean normal to $P$ pointing outward from $\mathrm{Hull}(\cals)$ and $\delta$ is a Euclidean normal to $Q$ in $P$ pointing outward from $\mathrm{Hull}(\cals_0)$, then $\eta_t = (1-t)\eta+t\delta$ is normal to a plane $P_t$ with $P_t\cap P = Q$.  Let $\bx_0\in Q$.  We have $\eta\cdot (\bo-\bx_0) > 0$, since $P$ separates $\cals$ from $\bo$ but does not contain $\bo$, so $\eta_t\cdot (\bo-\bx_0)>0$ for all small enough $t$.  One checks directly that for all $t>0$, $\eta_t\cdot (\bx-\bx_0) > 0$, $\eta_t\cdot (\bs-\bx_0) = 0$ for all $\bs\in\cals_0\cap Q$ and $\eta\cdot(\bs-\bx_0) < 0$ for all $\bs\in\cals_0-Q$.  The last equation holds for all $\bs\in\cals-Q$ if $t$ is small enough.  This follows from the fact below.

\begin{fact}\label{small neighborhood}  For $\epsilon>0$ and an affine plane $P$ with $P\cap U_{n+1}$ compact, if $P_t\to P$ (in the sense of Lemma \ref{plane convergence}) then $P_t\cap U_{n+1}$ lies in the $\epsilon$-neighborhood of $P\cap U_{n+1}$ for all small enough $t$.\end{fact}

The proof is a short exercise that also uses convexity of $U_{n+1}$.  Since $\cals$ is locally finite there exists $\epsilon>0$ such that $\cals-\cals_0$ is outside the $\epsilon$-neighborhood of $P\cap U_{n+1}$.  Fact \ref{small neighborhood} therefore implies that for small enough $t$ it lies in the same half-space determined by $P_t$ as by $P$.

We note also that the face $Q\cap F = P_t\cap\mathrm{Hull}(\cals)$ is also a face of $\mathrm{Hull}(\cals)$.  Since $Q$ is an arbitrary support plane for $F$ in $P$, the lemma's final assertion holds.\end{proof}

\begin{proof}[Proof of Proposition \ref{geometric dual}]  Note that $\cals_0=\{\bs_0,\hdots,\bs_l\}$ is finite since its members lie in a metric sphere.  Since $V$ is $k$-dimensional then applying Fact \ref{equations n stuff} we find that $\cals_0$ spans an $(n+1-k)$-dimensional subspace $W_0$ of $\mathbb{R}^{n+1}$.  For fixed $\bv\in V$, if $c_0 = \bv\circ\bs_0$ then $\bv\circ \bs_i = c_0$ for all $i>0$ by the definition of the hyperbolic metric below (\ref{geodesic ray}).  For $\bs\in\cals-\cals_0$ we have:
$$ -\bv\circ\bs = \cosh d_H(\bv,\bs) > \cosh d_H(\bv,\bs_0) = -c_0 $$
Thus $\cals$ is contained in the half-space $B_0 = \{\bx\circ\bv \leq c_0\}$, and its bounding hyperplane $P$ is a support plane for $\mathrm{Hull}(\cals)$ containing $\cals_0$.  The face $F = P\cap\mathrm{Hull}(\cals)$ contains $\cals_0$ and is contained in the $(n-k)$-dimensional affine space $P\cap W_0$, so has dimension $n-k$.  Since $\bo$ is in the opposite half-space $B$ to $B_0$, $F$ is visible.

The translate $P - \bs_0 = \{\bx\circ\bv = 0\}$ of $P$ is a space-like vector subspace, since its normal $\bv$ is time-like.  Therefore by Lemma \ref{rotate space}, $F=\mathrm{Hull}(\cals_0)$ is compact.

Now suppose $F = P\cap\mathrm{Hull}(\cals)$ for a support plane $P$ parallel to a space-like subspace.  By Lemma \ref{rotate space} again $F = \mathrm{Hull}(\cals_0)$ is compact, where $\cals_0 = P\cap\cals$.  Reversing the argument above shows that the normal $\bv\in\mathbb{H}^n$ to $V$ is equidistant from the points of $\cals_0$ and closer to them than any others; hence $V = \bigcap_{\bs\in\cals_0} V_{\bs}$ is a non-empty (since it contains $\bv$) face of the Voronoi tessellation.\end{proof}

\begin{lemma}\label{contravariance}  Geometric duality is contravariant with respect to inclusion of faces: for locally finite $\cals\subset\mathbb{H}^n$, if $V'$ is a face of a Voronoi cell $V$ then $C_V$ is a face of the geometric dual $C_{V'}$ to $V'$; and every face of $C_V$ is of the form $C_{V''}$ for some Voronoi cell $V''$ containing $V$.\end{lemma}

\begin{proof}  Let $C_V$ have vertex set $\cals_0 = \{\bs_0,\hdots,\bs_l\}$, so $V = \bigcap_{i=0}^l V_{\bs_i}$.  A face $V'$ of $V = \bigcap_{i=0}^l V_{\bs_i}$ is of the form $V\cap\bigcap_{i=l+1}^k V_{\bs_i}$ for an additional collection $\{\bs_{l+1},\hdots,\bs_k\}\subset \cals$.  Let $\cals_0' = \{\bs_0,\hdots,\bs_k\}$.  Then $C_{V'} = r_n(\mathrm{Hull}(\cals_0'))$ by Proposition \ref{geometric dual}, and the hyperplane $P$ from the proof of the proposition is in particular a supporting hyperplane for $\mathrm{Hull}(\cals_0')$ with $P\cap\mathrm{Hull}(\cals_0') = P\cap\mathrm{Hull}(\cals) = \mathrm{Hull}(\cals_0)$.  Hence $C_V$ is a face of $C_{V'}$.

For a face $C_0$ of $C_V$, an argument from the proof of Proposition \ref{the finite case} produces a face $F_0$ of $F$ such that $r_n(F_0) = C_0$, where $F$ is the visible face of $\mathrm{Hull}(\cals)$ mapping to $C_V$.  Thus $F_0 = \mathrm{Hull}(\cals_0')$ for some $\cals_0'\subset\cals_0$, so $C_0$ is the geometric dual to the Voronoi cell $V'' = \bigcap_{\bs\in\cals_0'} V_{\bs}$ containing $V$.\end{proof}

Proposition \ref{geometric dual} implies for a locally finite set $\cals\subset\mathbb{H}^n$ that the collection of Delaunay cells that are geometrically dual to Voronoi cells has the following description:
$$ \{r_n(F)\,|\,F = P\cap\mathrm{Hull}(\cals)\ \mbox{for a support plane $P$ parallel to a space-like subspace}\} $$ 
We prove below that it is a polyhedral complex in the standard sense of \cite[Dfn.~2.1.5]{DeLoRS}, and characterize it by an empty metric circumspheres condition.

\begin{theorem}\label{the real geometric dual} \TheRGD \end{theorem}

\begin{remark}\label{bad dual}  For $\cals$ as in the main example of Section \ref{bad example}, the geometric dual to the Voronoi tessellation contains every face but $C_{\infty}$, since these have metric circumspheres.  It is thus not locally finite, at $p_0$ in particular, and its underlying space is not closed nor convex.\end{remark}

\begin{proof}[Proof of Theorem \ref{the real geometric dual}]  The description of the geometric dual complex above follows from Proposition \ref{geometric dual}.  By Corollary \ref{black sox}, every geometric dual cell is a compact, convex polyhedron.  For such a cell $C$ write $C=r_n(F)$ for the face $F=P\cap\mathrm{Hull}(\cals)$ given by Proposition \ref{geometric dual}, where $P$ is a support plane parallel to a space-like subspace hence with $S=P\cap\mathbb{H}^n$ a metric sphere (Lemma \ref{classification of spheres}).  Corollary \ref{black sox} implies $C$ is the closed convex hull of $S\cap\cals$.

The half-space $B$ bounded by $P$ and containing $\bo$ intersects $\mathbb{H}^n$ in a ball bounded by $S$ (recall Lemma \ref{convex component}), which contains $r_n(F)$ by Lemma \ref{ratatouille}, and $(B\cap\mathbb{H}^n)\cap\cals = S\cap\cals$.  On the other hand, a metric sphere $S$ that intersects $\cals$ and bounds an empty ball is of the form $P\cap\mathbb{H}^n$ for a support plane $P$ bounding a half-space $B$ containing $\bo$ with $B\cap\cals=P\cap\cals$.  Proposition \ref{geometric dual} therefore implies that $r_n(F)$ is a geometric dual cell, where $F = P\cap\mathrm{Hull}(\cals)$.

That a face of a geometric dual cell is itself a geometric dual cell is one of the assertions of Lemma \ref{contravariance}.  To prove that geometric dual cells intersect in a face of each we follow the proof of the corresponding assertion of Proposition \ref{the finite case}.  It applies without only one alteration: replace the appeal to Lemma \ref{good visibility} with one to Proposition \ref{geometric dual}, and note that no affine plane intersecting $U_{n+1}$ and parallel to a space-like subspace contains $\bo$. \end{proof}

The Voronoi tessellation's geometric dual may be a proper subcomplex of the Delaunay tessellation   even in the simplest possible case of three points in $\mathbb{H}^2$.

\begin{example}\label{no dual} \textbf{The simplest case.}  Figure \ref{three point Vor} illustrates (in the upper half-plane model) the Voronoi and Delaunay tessellations determined by three points in $\mathbb{H}^2$, configured as in Figure \ref{three pt circ}.  In each case the Delaunay triangle spanned by $x$, $y$, and $z$ is shaded, with its edges dashed.  The edges of the Voronoi tessellation are in bold.


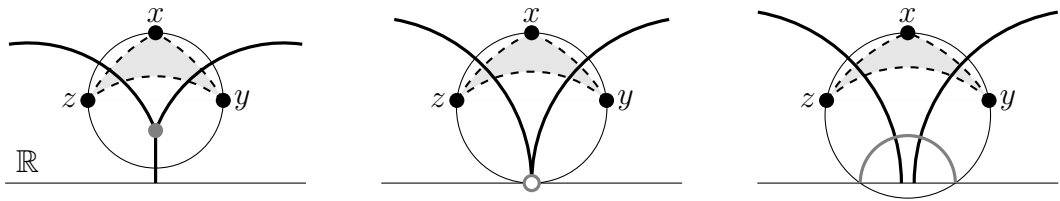
\begin{figure}
\begin{tikzpicture}

\begin{scope}[xshift=-5cm]
\draw (-2,0) -- (2,0);
\node [above right] at (-2,0) {$\mathbb{R}$};

\fill [opacity=0.1] (0.9,1.1) -- (0,2) -- (-0.9,1.1);
\fill [opacity=0.1] (0.9,1.1) arc (28.811:61.189:2.282);
\draw [thick, dashed] (0.9,1.1) arc (28.811:61.189:2.282);
\fill [opacity=0.1] (-0.9,1.1) arc (151.189:118.811:2.282);
\draw [thick, dashed] (-0.9,1.1) arc (151.189:118.811:2.282);
\fill [color=white] (0.9,1.1) arc (50.711:129.289:1.421);
\draw [thick, dashed] (0.9,1.1) arc (50.711:129.289:1.421);

\draw (0,1.1) circle [radius=0.9];

\draw [very thick] (0,0) -- (0,0.7);
\draw [very thick] (0,0.7) arc (158.94:82:1.818);
\draw [very thick] (0,0.7) arc (21.06:98:1.818);

\fill [color=gray] (0,0.7) circle [radius=0.1];

\fill (0.9,1.1) circle [radius=0.1];
\node [right] at (0.9,1.1) {$y$};
\fill (-0.9,1.1) circle [radius=0.1];
\node [left] at (-0.9,1.1) {$z$};
\fill (0,2) circle [radius=0.1];
\node [above] at (0,2) {$x$};

\end{scope}

\begin{scope}
\draw (-2,0) -- (2,0);

\fill [opacity=0.1] (0.995,1.1) -- (0,2) -- (-0.995,1.1);
\fill [opacity=0.1] (0.995,1.1) arc (30.075:65.665:2.195);
\draw [thick, dashed] (0.995,1.1) arc (30.075:65.665:2.195);
\fill [opacity=0.1] (-0.995,1.1) arc (149.925:114.335:2.195);
\draw [thick, dashed] (-0.995,1.1) arc (149.925:114.335:2.195);
\fill [color=white] (0.995,1.1) arc (47.869:132.131:1.483);
\draw [thick, dashed] (0.995,1.1) arc (47.869:132.131:1.483);

\draw (0,1) circle [radius=1];

\draw [very thick] (0,0) arc (180:100:2.211);
\draw [very thick] (0,0) arc (0:80:2.211);

\fill [color=white] (0,0) circle [radius=0.1];
\draw [very thick, color=gray] (0,0) circle [radius=0.1];
\fill (0.995,1.1) circle [radius=0.1];
\node [right] at (0.995,1.1) {$y$};
\fill (-0.995,1.1) circle [radius=0.1];
\node [left] at (-0.995,1.1) {$z$};
\fill (0,2) circle [radius=0.1];
\node [above] at (0,2) {$x$};

\end{scope}

\begin{scope}[xshift=5cm]
\draw (-2,0) -- (2,0);

\draw (0,0.9) circle [radius=1.1];

\fill [opacity=0.1] (1.08,1.1) -- (0,2) -- (-1.08,1.1);
\fill [opacity=0.1] (1.08,1.1) arc (31.01:69.39:2.136);
\draw [thick, dashed] (1.08,1.1) arc (31.01:69.39:2.136);
\fill [opacity=0.1] (-1.08,1.1) arc (148.99:110.61:2.136);
\draw [thick, dashed] (-1.08,1.1) arc (148.99:110.61:2.136);
\fill [color=white] (1.08,1.1) arc (45.48:134.52:1.54);
\draw [thick, dashed] (1.08,1.1) arc (45.48:134.52:1.54);

\fill (1.08,1.1) circle [radius=0.1];
\node [right] at (1.08,1.1) {$y$};
\fill (-1.08,1.1) circle [radius=0.1];
\node [left] at (-1.08,1.1) {$z$};
\fill (0,2) circle [radius=0.1];
\node [above] at (0,2) {$x$};

\draw [very thick] (0.085,0) arc (180:100:2.31);
\draw [very thick] (-0.085,0) arc (0:80:2.31);
\draw [very thick, color=gray] (0.632,0) arc (0:180:0.632);

\end{scope}

\end{tikzpicture}
\caption{Delaunay and Voronoi tessellations determined by three points in $\mathbb{H}^2$.}
\label{three point Vor}
\end{figure}

In the left case the Delaunay tessellation and the geometric dual complex coincide.  In particular, the Delaunay triangle is the geometric dual to the Voronoi vertex: the grey dot.  In the middle and on the right, the Voronoi tessellation has no vertex and the Delaunay triangle has no geometric dual; instead, the geometric dual to the Voronoi tessellation has cells $x$, $y$, $z$, and the two edges containing $x$.  Recall from Example \ref{three pt sphere} that the Delaunay triangles' circumspheres are horospherical and metric in these respective cases. \end{example}

\begin{remark}  For $x$, $y$, and $z$ from Example \ref{no dual} let $\bx = I(x)$, $\by = I(y)$ and $\bz=I(z)$, where $I$ is an isometry from the upper half-plane to the hyperboloid model of $\mathbb{H}^2$.  In all cases the planes in $\mathbb{R}^3$ containing the Voronoi edges $V_{\bx}\cap V_{\by}$ and $V_{\bx}\cap V_{\bz}$ intersect in a line.  In the left-hand case this line is spanned by the Voronoi vertex $\bu = V_{\bx}\cap V_{\by}\cap V_{\bz}\in\mathbb{H}^2$, which by Proposition \ref{geometric dual} is the geometric dual to the Delaunay triangle.  In the other two cases we may regard a spanning vector $\bu$ for the line of intersection as an analog of this vertex.

We further strengthen this analogy by observing that the planes containing $V_{\bx}\cap V_{\by}$ and $V_{\bx}\cap V_{\bz}$ are respectively Lorentz-orthogonal to $\bx-\by$ and $\bx-\bz$, so in all cases such a vector $\bu$ has the property that $\bu\circ(\bx-\by) = \bu\circ(\bx-\bz) = 0$.  Therefore the Lorentz-orthogonal complement $\bu^{\perp}$ to $\bu$ is parallel to the support plane for $\mathrm{Hull}(\{\bx,\by,\bz\})$ in $\mathbb{R}^3$ in all cases, so $\bu$ is ``dual'' to the Delaunay triangle in the sense of Proposition \ref{geometric dual}.  (Also, it follows from this and Lemma \ref{classification of spheres} that $\bu$ is light-like in the middle case and space-like in the right-hand case of Example \ref{no dual}.)

This example suggests that the geometric duality correspondence of Proposition \ref{geometric dual} might be extended in some cases by allowing Voronoi ``cells'' that lie outside $\mathbb{H}^2$ in, say, projective space.  While we will not pursue it here, this seems worth further study.\end{remark}

\section{Tessellating manifolds}\label{manifolds}

In this section we will construct Delaunay tessellations of finite subsets of finite-volume hyperbolic manifolds.  The first basic observation is that invariant sets have invariant hulls.

\begin{lemma}\label{invariance'n'stuff}  Suppose $\cals\subset\mathbb{H}^n$ is invariant under a subgroup $\Gamma$ of $\mathit{SO}^+(1,n)$.  Then $\mathrm{Hull}(\cals)$ is also $\Gamma$-invariant, and the $\Gamma$-action takes faces to faces and commutes with $r_n$.  If $\cals$ is locally finite its Voronoi tessellation and its geometric dual are $\Gamma$-invariant.\end{lemma}

\begin{proof}  Every $\gamma\in\Gamma$, being linear, takes the convex hull of $\cals$ into itself; being invertible, it is a self-bijection of the convex hull; and being continuous and invertible it is a self-homeomorphism of $\mathrm{Hull}(\cals)$.  Again since it is linear, $\gamma$ takes affine planes to affine planes and hence faces to faces.  Since $\gamma$ is in $\mathit{SO}^+(1,n)$ it takes $\mathbb{H}^n$ to itself, so since it takes rays through $\bo$ to rays through $\bo$ it commutes with $r_n$.

If $\cals$ is locally finite its Voronoi tessellation is defined (see Section \ref{into Voronoi}).  Since the definition is  in terms of distance and $\Gamma$ acts by isometries it preserves the Voronoi tessellation.  It follows from Proposition \ref{geometric dual} that the geometric dual to the Voronoi tessellation is also preserved.\end{proof}

Recall from the introduction that a \textit{lattice} in $\mathit{SO}^+(1,n)$ is a discrete subgroup with a finite-volume fundamental domain.  Lattice-invariance imposes strong constraints on the nature of faces of $\mathrm{Hull}(\cals)$.  One reason is the following basic fact:

\begin{fact}\label{limit of lattice}  The limit set of a lattice $\Gamma<\mathit{SO}^+(1,n)$ is the entire sphere at infinity of $\mathbb{H}^n$.\end{fact}

For standard material on limit sets see eg.~Ch.~12 of \cite{Ratcliffe}. In particular, Fact \ref{limit of lattice} is Theorem 12.2.13 there.  Below we interpret these notions in the context of the hyperboloid model:

\begin{definition}\label{at infinity}  The \textit{sphere at infinity} of $\mathbb{H}^n$ is the projectivized positive light cone $p(L^+)$ in $\mathbb{R}P^n$.  Here $p\co\mathbb{R}^{n+1}-\{\bo\}\to\mathbb{R}P^n$ is the quotient map.  The \textit{limit set} of $\cals\subset\mathbb{H}^n$ is the set of accumulation points of $p(\cals)$ in $p(L^+)$, where $L^+=\{\bu\,|\,\bu\circ\bu=0,u_0>0\}$ is the positive light cone.  The \textit{limit set} of $\Gamma<\mathit{SO}^+(1,n)$ is the limit set of $\Gamma.\bx$ for some (any) fixed $\bx\in \mathbb{H}^n$.\end{definition}

The following is an exercise in Lorentzian geometry.

\begin{fact}\label{limit sets}  Suppose $P$ is an affine subspace of $\mathbb{R}^{n+1}$ intersecting $\mathbb{H}^n$, and let $V$ be the vector subspace parallel to $P$ and $B$ a half-space bounded by $P$.\begin{itemize}
  \item  If $V$ is space-like and $B\cap\mathbb{H}^n$ is convex then its limit set is empty.
  \item  If $V$ is light-like and $B\cap\mathbb{H}^n$ is convex then its limit set is the singleton $p(V^{\perp}\cap L^+)$.
  \item  If $V$ is time-like then the limit set of $B\cap\mathbb{H}^n$ is a hemisphere bounded by $p(V\cap L^+)$.\end{itemize}\end{fact}

\begin{lemma}\label{no time}  Suppose $\cals\subset\mathbb{H}^n$ is locally finite and $\Gamma$-invariant for a lattice $\Gamma$ of $\mathit{SO}^+(1,n)$.  For any $n$-dimensional affine subspace $P$ that bounds a half-space $B$ containing $\mathrm{Hull}(\cals)$, the vector subspace parallel to $P$ is not time-like.  If $P$ intersects $\mathbb{H}^n$ then $\bo\notin B$.\end{lemma}

\begin{proof}  Since $\cals$ is $\Gamma$-invariant the limit set of $\mathrm{Hull}(\cals)$ contains that of $\Gamma$ and thus is the entire sphere at infinity.  The fact above thus implies the result (recall also from Lemma \ref{convex component} that $B\cap\mathbb{H}^n$ is convex if and only if $\bo\in B$.)\end{proof}

\begin{corollary}\label{empty spheres}  Suppose $\cals\subset\mathbb{H}^n$ is locally finite and invariant under a lattice $\Gamma$ of $\mathit{SO}^+(1,n)$.  Every support plane for $\mathrm{Hull}(\cals)$ separates $\mathrm{Hull}(\cals)$ from $\bo$. In particular every face of $\mathrm{Hull}(\cals)$ is visible, and every visible point lies in a visible face.  For a face $F$ and a support plane $P$ with $F=P\cap\mathrm{Hull}(\cals)$, if $B$ is the half-space bounded by $P$ that contains $\bo$ then $r_n(F)\subset B$.\end{corollary}

\begin{proof}  That support planes separate $\mathrm{Hull}(\cals)$ from $\bo$ follows directly from the final assertion of Lemma \ref{no time}.  (Note that each support plane contains a face of $\mathrm{Hull}(\cals)$, hence intersects $U_{n+1}$ and hence $\mathbb{H}^n$.)  Since each face lies in a support plane, it follows directly that each face is visible. Since each visible point is contained in a face, it lies in a visible face.  The final assertion above follows directly from Lemma \ref{ratatouille}.\end{proof}

\begin{corollary}\label{all the points}  If $\cals\subset\mathbb{H}^n$ is locally finite and invariant under a lattice $\Gamma$ of $\mathit{SO}^+(1,n)$ then $r_n(\mathrm{Hull}(\cals)) = \mathbb{H}^n$, and $\mathbb{H}^n=\bigcup \{r_n(F)\,|\,F\ \mbox{is a visible face of}\ \mathrm{Hull}(\cals)\}$.\end{corollary}

\begin{proof}  We note first that $r_n(\mathrm{Hull}(\cals))$ is closed, since if it were not then Lemma \ref{convex hull boundary} would supply a time-like $n$-dimensional subspace $V$ with $\cals$ on one side of it.  This violates Lemma \ref{no time}.  Since $r_n(\mathrm{Hull}(\cals))$ is closed, if it were not all of $\mathbb{H}^n$ then Lemma \ref{HB} would supply a subspace $V$ as above, again violating Lemma \ref{no time}.

The $r_n$-image of $\mathrm{Hull}(\cals)$ is the image of its set of visible points, so the result holds since each such point is contained in a visible face by Corollary \ref{empty spheres}.\end{proof}

\subsection{A brief interruption on horospheres and horoballs}  To understand the restrictions that lattice-invariance places on horospherical circumspheres, we need a little more information about horospheres. The results here are standard (with the possible exception of Lemma \ref{Euclid's planes}), but we record them for completeness.

\begin{lemma}\label{bad balls}  Suppose for a light-like subspace $V$ of $\mathbb{R}^{n+1}$ and $\bx_0\in\mathbb{R}^{n+1}$ that $P = V+\bx_0$ intersects $\mathbb{H}^n$.  For the unique $\bu\in V^{\perp}$ such that $P=\{\bx\circ\bu=-1\}$ and any $k<0$, let $P_k = \{\bx\circ\bu = k\}$.  For any $\bx\in S=P\cap\mathbb{H}^n$, the unique closest point of $S_k = P_k\cap\mathbb{H}^n$ to $\bx$ is  $\gamma_{\bx}(t_k)$, where $\gamma_{\bx}$ is the geodesic ray from Lemma \ref{convex component} and $t_k=\ln (-1/k)$.\end{lemma}

\begin{remark}\label{busemann horo}  It follows from Lemma \ref{bad balls} that for $k\geq-1$ each $S_k$ above satisfies the classical definition of a horosphere as a level set of the \textit{Busemann function} of $\gamma_{\bx}$, see eg.~\cite[\S II.8]{BrH}.\end{remark}

\begin{proof} A computation shows the unique point of intersection between $\gamma_{\bx}$ and $S_k$ is $\gamma_{\bx}(t_k)$.  For $\bx$ as above and arbitrary $\by\in S-\{x\}$, consider:
$$ -\bx\circ\gamma_{\by}(t) = e^{-t}(-\bx\circ\by) + \sinh t(-\bx\circ\bu) > \cosh t, $$
This is because $\bx\circ\bu = -1$ by hypothesis, and $\bx\circ\by\leq -1$ since both are in $\mathbb{H}^n$ (recall (\ref{x circ y})).  It yields the basic fact that $d_H(\bx,\gamma_{\by}(t)) > t = d_H(\by,\gamma_{\by}(t))$ for each $t\in\mathbb{R}$.  In particular, for the unique point $\gamma_{\by}(t_k)$ of intersection between $\gamma_{\by}$ and $S_k$, $d_H(\bx,\gamma_{\by}(t_k))>t_k$.

The lemma follows directly from the claim that the map $\by\mapsto \gamma_{\by}(t_k)$ is onto $S_k$.  This in turn follows from the facts that for any $\bz\in S_k$, the geodesic $\gamma_{\bz}$ has unique point $\by_{\bz} = \gamma_{\bz}(-t_k)$ with $S$, and that $\gamma_{\by_{\bz}}(t) = \gamma_{\bz}(t-t_k)$ for each $t\in\mathbb{R}$; in particular, $\gamma_{\by_{\bz}}(t_k) =\bz$.\end{proof}

\begin{lemma}\label{Euclid's balls}  For any horosphere $S$ of $\mathbb{H}^n$ there is a Euclidean isometry $\mathbb{R}^{n-1}\to S$.\end{lemma}

\begin{proof}  Suppose $V$ is an $n$-dimensional light-like subspace of $\mathbb{R}^{n+1}$ and $\bx_0$ is a vector such that $P=V+\bx_0$ intersects $\mathbb{H}^n$.  Without loss of generality we may take $\bx_0\in\mathbb{H}^n$. Fix $\bu\in V^{\perp}$ such that $\bx_0\circ\bu = -1$, and let $V_0 = \bx_0^{\perp}\cap V$.  This is a space-like subspace of $V$ (\cite[Theorem 3.1.5]{Ratcliffe}) of dimension $(n-1)$ since $\bu\notin V_0$.  Thus $V$ is spanned by $V_0$ and $\bu$.

Writing an arbitrary element of $V$ as $t\bu+\bv$ for $t\in\mathbb{R}$ and $\bv\in V_0$, consider the Lorentzian norm of $\bx_0+(t\bu+\bv)\in P$.\begin{align}\label{go circ yourself}
 [\bx_0+(t\bu+\bv)]\circ[\bx_0+(t\bu+\bv)] = \bx_0\circ\bx_0 - 2t + \bv\circ\bv \end{align}
For each $\bv\in V_0$ it follows that $\bx_0+(t\bu+\bv)\in\mathbb{H}^n$ if and only if $t = \frac{1}{2}\left(1+\bx_0\circ\bx_0 + \bv\circ\bv\right)$.  Let $k = 1+\bx_0\circ\bx_0$ and $\|\bv\|^2 = \bv\circ\bv$.  The map
$$ F(\bv) = \bx_0 + \bv + \frac{1}{2}(k+\|\bv\|^2)\bu $$
is therefore a homeomorphism from $V$ to the horosphere $P\cap\mathbb{H}^{n-1}$.  Using basic calculus one easily checks that for $\bw\in V_0$, $dF_{\bv}(\bw) = \bw + 2(\bv\circ\bw)\bu$, so $F$ is an isometry since $\bu\in V^{\perp}$.  But $V_0$ is space-like, so $\circ|_{V_0}$ is positive-definite and $(V_0,\circ)$ is isometric to Euclidean $(n-1)$-space.\end{proof}

\begin{lemma}\label{intersect ball}  For $\bu\in L^+ = \{\bx\,|\,\bx\circ\bx =0, x_0>0\}$ and a space-like vector $\bv$ such that $\bu\notin V = \bv^{\perp}$, the totally geodesic subspace $V\cap\mathbb{H}^n$ has compact intersection with the horoball $B \cap\mathbb{H}^n$, where $B = \{\bx\circ\bu\geq -1\}$.\end{lemma}

\begin{proof}  For $\bw = s\bu+t\bv$, $\bw\circ\bw = 2st\bu\circ\bv+t^2\bv\circ\bv$.  Since $\bu\circ\bv\neq 0$ by hypothesis, $\bw\circ\bw=-1$ if and only if $s = \frac{-1-t^2\bv\circ\bv}{2t\bu\circ\bv}$.  Thus any $t\in\mathbb{R}-\{0\}$ determines $s(t)$ such that $\bw(t) = s(t)\bu+t\bv$ has $\bw(t)\circ\bw(t)=-1$.  The function $s(t)$ has opposite signs on the two components of $\mathbb{R}-\{0\}$, and $s(t)\to\infty$ as $t\to 0$.  Since $u_0>0$ it follows that $\bw(t)\in\mathbb{H}^n$ when $s(t)>0$.

Fix some $t$ with $s(t)>0$, so $\bw(t)\in\mathbb{H}^n$.  If $\bx\in V\cap B$ then $\bx\circ\bw(t) \geq -s(t)$, so by Lemma \ref{convex component} $(V\cap B)\cap\mathbb{H}^n$ is contained in the ball of radius $\cosh^{-1} (s(t))$ about $\bw(t)$.\end{proof}

\begin{lemma}\label{intersect balls}  For linearly independent vectors $\bu$ and $\bu'$ in $L^+$, the horoball intersection $(B\cap B')\cap \mathbb{H}^n$ is compact, where $B = \{\bx\circ\bu\geq -1\}$ and $B'=\{\bx\circ\bu' \geq -1\}$.  For the horosphere $S$ determined by $\bu$, if $S\cap B'\neq\emptyset$ then $S\cap B'=S\cap U$ for a hyperbolic ball $U$ centered in $S$.  As $\bu'\to\bu$ in the Euclidean sense, eventually $S\cap B'\neq\emptyset$ and the radius of $U$ increases without bound.\end{lemma}

\begin{proof}  The span of $\bu$ and $\bu'$ intersects $\mathbb{H}^n$ in the (non-linearly reparametrized) geodesic $\gamma(a) = a\bu+b(a)\bu'$, $a>0$, where $b(a) = -1/(2a\bu\circ\bu')$.  A computation shows that $\bx_0 = \gamma(1/2)$ is the unique point of intersection between $\gamma$ and $S$.  For arbitrary $\bx$ in $S$ we have $\bx\in B' \Leftrightarrow \bx\circ\bu'\geq -1$; i.e.~if and only if
$$\bx\circ\bx_0 = \frac{-1}{2} + \frac{-1}{\bu\circ\bu'}\bx\circ\bu' \geq \frac{-1}{2}+\frac{1}{\bu\circ\bu'}.$$
Let $r_0$ be the right-hand quantity above.  Then $S\cap B'\neq\emptyset$ if and only if $r_0\leq -1$.  If so then $S\cap B' = S\cap U$, where $U$ is the ball of radius $\cosh^{-1}(-r_0)$ centered at $\bx_0$.  Note that $r_0\to-\infty$ as $\bu'\to \bu$, since then $\bu'\circ\bu\to 0$ (from below by (\ref{x circ y})) by continuity of $\circ$.

It remains only to note that any $\bx\in B\cap B'$ also satisfies the inequality above, so $B\cap B'$ is a closed subset of $U$ hence compact.\end{proof}

\begin{lemma}\label{intersect ballssss}  For a sequence $\{\bu_n\}\to\bu$ of time-like vectors in $\mathbb{R}^{n+1}$ converging to a light-like vector, and $s_n\to -1$, let $P_n = \{\bx\circ\bu_n=s_n\}$ and $B_n = \{\bx\circ\bu_n\geq s_n\}$.  If the horosphere $S = P\cap\mathbb{H}^n$ determined by $\bu$ (where $P=\{\bx\circ\bu=-1\}$) contains a point $\bx$ with $\bx\circ\bu_n\to -1$, and $B_n\cap S$ contains no sequence with unbounded distance to $S-B_n$ as $n\to\infty$, then
$$ \lim_{n\to\infty} \frac{\bu_n\circ\bu_n}{\bu\circ\bu_n} = 2 $$\end{lemma}

\begin{proof}  The geodesic $\gamma$ of $\mathbb{H}^n$ described below is from Lemma \ref{convex component}, with $\frac{\sqrt{-\bu_n\circ\bu_n}}{-\bu\circ\bu_n}\bu$ in the role of $\bu$ there and $\gamma(0)=\bu_n/\sqrt{-\bu_n\circ\bu_n}$ in the role of $\bx$. (The scale factors ensure correct pairings.)
$$\gamma(t) = e^{-t}\frac{\bu_n}{\sqrt{-\bu_n\circ\bu_n}} + \sinh t\frac{\sqrt{-\bu_n\circ\bu_n}}{-\bu\circ\bu_n}\,\bu$$
$P_n\cap\mathbb{H}^n$ is a metric sphere centered at $\gamma(0)$ (by Lemma \ref{classification of spheres}), and $B_n\cap\mathbb{H}^n$ is a metric ball with the same center (Lemma \ref{convex component}), each with radius $\cosh^{-1}(-s_n/\sqrt{-\bu_n\circ\bu_n})$.  Since $s_n\to -1$ and $\bu_n\circ\bu_n\to 0$ (by continuity of $\circ$), the radius approaches infinity as $\bu_n\to\bu$.

A computation reveals that $\gamma(t_n)$ is the unique point of intersection between $\gamma$ and $S$, where $t_n= \ln\frac{-\bu\circ\bu_n}{\sqrt{-\bu_n\circ\bu_n}}$.  Since $\gamma = \gamma_{\gamma(0)}$, Lemma \ref{bad balls} implies that $\gamma(t_n)$ is the closest point of $S$ to $\gamma(0)$, in particular closer than $\bx$.  Therefore:
$$ \gamma(t_n)\circ\bu_n = \frac{1}{2}\left[\bu\circ\bu_n - \frac{\bu_n\circ\bu_n}{\bu\circ\bu_n}\right] 
 \geq \bx\circ\bu_n \quad\Rightarrow\quad \frac{\bu_n\circ\bu_n}{\bu\circ\bu_n} \leq -2\bx\circ\bu_n+\bu\circ\bu_n$$
The right-hand inequality above implies that $\limsup_{n\to\infty} \frac{\bu_n\circ\bu_n}{\bu\circ\bu_n}\leq2$.

If $\gamma(t_n)$ is not in $B_n$ then $\gamma(t_n)\circ \bu_n< s_n$.  Supposing that this occurs for all but finitely many $n$, rearranging as above we infer that $\liminf_{n\to\infty} \frac{\bu_n\circ\bu_n}{\bu\circ\bu_n} \geq 2$, and the result holds.  We therefore restrict attention to an infinite subsequence on which we suppose that $\gamma(t_n)\in B_n$.

For $\by\in S$ (so $\by\circ\bu=-1$) we have $\by\in B_n$ if and only if $\by\circ\bu_n\geq s_n$; i.e.~if and only if
$$ \gamma(t_n)\circ\by = \frac{\by_n\circ\bu_n}{-\bu\circ\bu_n} - \frac{1}{2}\left[1-\frac{-\bu_n\circ\bu_n}{(\bu\circ\bu_n)^2}\right] \geq -\frac{1}{2} + \frac{-1}{2\bu\circ\bu_n}\left[2s_n+\frac{\bu_n\circ\bu_n}{\bu\circ\bu_n}\right]  $$
Thus $B_n\cap S = U_n\cap S$ where $U$ is a metric ball centered at $\gamma(t_n)$ with radius the inverse hyperbolic cosine of the right-hand quantity above.  By hypothesis this radius remains bounded, so since $\bu\circ\bu_n\to 0$ the bracketed quantity on the right approaches $0$ as $n\to\infty$.  The Lemma follows, since $s_n\to -1$.
\end{proof}

\begin{lemma}\label{Euclid's planes}  Let $V$ be an $n$-dimensional light-like subspace of $\mathbb{R}^{n+1}$, $\bx_0$ such that $P = V+\bx_0$ intersects $\mathbb{H}^n$, and $\bu\in V^{\perp}$ such that $\bx_0 \circ \bu = -1$.  For a codimension-one affine subspace $Q$ of $P$, exactly one of the following holds:
\begin{itemize} \item $Q$ contains a translate of $\bu$, and $Q\cap\mathbb{H}^n$ is a totally geodesic hyperplane in the Euclidean metric on $P\cap\mathbb{H}^n$; or
\item  the half-space $P^-=\{\bq - t\bu\,|\, \bq\in Q\ \mbox{and}\ t\geq 0\}$ of $P$ bounded by $Q$ has compact intersection with $U_{n+1} = \{\bx\,|\,\bx\circ\bx\leq -1,\ x_0>0\}$.
\end{itemize}\end{lemma}

The intuition here is from conic sections: since $P\cap\mathbb{H}^n$ is asymptotic to the paraboloid $P\cap L$, $Q\cap U_{n+1}$ cuts off two non-compact pieces if and only if $Q$ is ``parallel to $L$''.

\begin{proof}  Let us first suppose that $Q$ contains a translate of $\bu$, and as in the proof of Lemma \ref{Euclid's balls}, suppose $\bx_0$ is Lorentz-orthogonal to $V_0 = V\cap(\{0\}\times\mathbb{R}^n)$.  It follows from (\ref{go circ yourself}) that
$$P\cap U_{n+1} = \{\bx_0 + \bv + t\bu\,|\, \bv\in V_0\ \mbox{and}\ t \geq 1 + \bx_0\circ\bx_0+\bv\circ\bv\}$$
Upon noting that $V$ is spanned by $V_0$ and $V^{\perp}$, it is thus clear that $Q\cap U_{n+1}$ is non-compact.  For any such $Q$, $Q-\bx_0$ is an affine subspace of $V$ that has non-trivial intersection with $V_0$.  Let $Q_0 = (Q-\bx_0)\cap V_0$.  This is a totally geodesic Euclidean subspace of $V_0$, with codimension $1$ in $Q$, that is mapped by $F$ (from Lemma \ref{Euclid's balls}) to $Q\cap\mathbb{H}^n$.

Suppose now that $Q$ contains no translate of $\bu$, and assume that $\bx_0\in Q$.  Then $Q_0=Q-\bx_0$ is a space-like subspace of $V$.  The orthogonal complement to $Q_0$ in $\mathbb{R}^{n+1}$ contains a time-like vector $\bw$, which we may take in $\mathbb{H}^n$ after scaling.  Let $W$ be the orthogonal complement to $\bw$, a space-like $n$-dimensional subspace containing $Q_0$, and let $R = W+\bx_0$, intersecting $P$ in $Q$.  Lemma \ref{convex component} implies that $B = \{\bx\circ\bw\geq\bx_0\circ\bw\}$ is that half-space bounded by $R$ that has compact intersection with $U_{n+1}$.  We claim this contains the half-space of $P$ above.

For $\bq\in Q$, $\bq\circ\bw = \bx_0\circ\bw$ since $\bw$ is by construction Lorentz-orthogonal to $Q_0 = Q-\bx_0$.  The claim, and hence the lemma, now follows from (\ref{x circ y}), which implies that $\bw\circ\bu < 0$.\end{proof}

\subsection{Finite-volume hyperbolic manifolds}\label{margulis}Consequences of Margulis' Lemma, a deep result on the geometry of hyperbolic manifolds, further constrain the faces of $\mathrm{Hull}(\cals)$ with horospherical circumspheres when $\cals$ is lattice-invariant.  In this brief section we will lay out the relevant results, all of which are standard, using the expository text by Benedetti--Petronio \cite{BenPet} for pinpoint citations.

For a discrete subgroup $\Gamma$ of $\mathit{SO}^+(1,n)$, let $M = \mathbb{H}^n/\Gamma$.  If $\Gamma$ is torsion-free then $M$ is a manifold, and if $\Gamma$ is a lattice then $M$ has finite volume.  For any $\epsilon>0$ let $M_{(0,\epsilon]}$ be the ``$\epsilon$-thin part'' of $M$, consisting of points at which the injectivity radius is at most $\epsilon$, and $M_{[\epsilon,\infty)}$, the ``$\epsilon$-thick part'', be the points where it is at least $\epsilon$.  The \textit{injectivity radius} of $M$ at $\bx$ is $\frac{1}{2}\min \{d_H(\tilde{x},g.\tilde{x})\}$, taken over all $g\in\Gamma-\{\mathit{id}\}$, for some $\tilde{x}\in\mathbb{H}^n$ projecting to $x$.

Margulis' Lemma asserts that for any $n\geq2$ there exists $\epsilon_n>0$, the \textit{$n$-dimensional Margulis constant}, such that if $0<\epsilon <\epsilon_n$ then for any torsion-free lattice $\Gamma<\mathit{SO}^+(1,n)$, every subgroup of $\Gamma$ consisting entirely of elements $g$ with $d_H(x,g.x) < \epsilon$ for some fixed $x\in \mathbb{H}^n$ is almost nilpotent (see \cite{BenPet}).
 
This implies that the $\epsilon$-thin part $M_{(0,\epsilon]}$ of $M = \mathbb{H}^n/\Gamma$ is a finite disjoint union of topological components classified by Theorem D.3.3 of \cite{BenPet}.  The \textit{cusps} of $M$, components of $M_{[0,\epsilon)}$ of type (2) in Theorem D.3.3 (see also the bottom of p.~150 in \cite{BenPet}), are most important here.  In particular, cusps are noncompact, whereas $M_{[\epsilon,\infty)}$ is compact \cite[Prop.~2.6]{BenPet} and so are components of $M_{(0,\epsilon]}$ of other types.

Cusps are associated to \textit{parabolic subgroups} of $\Gamma$.  An isometry $\gamma\in\mathit{SO}^+(1,n)$ of $\mathbb{H}^n$ is \textit{parabolic} if it has no time-like eigenvector and a unique light-like eigenvector $\bu$ with eigenvalue $1$.  Its \textit{fixed point} is $p(U\cap L^+)$, where $U = \mathrm{span}(\bu)$ and $p\co\mathbb{R}^{n+1}-\{\bo\}\to \mathbb{R}P^n$ is projectivization (recall Definition \ref{at infinity}).  A non-trivial group of isometries is parabolic if all its non-trivial elements are.  It is a basic fact that all elements of such a group share a fixed point.

\begin{lemma}\label{invariant horoballs}  Suppose $\Gamma<\mathit{SO}^+(1,n)$ is a torsion-free lattice, and $L$ is a cusp of $M = \mathbb{H}^n/\Gamma$.  For any component $\tilde{L}$ of $\pi^{-1}(L)$, where $\pi\co\mathbb{H}^n\to M$ is the quotient map, there is a parabolic subgroup $\Gamma_1$ of $\Gamma$ stabilizing $\tilde{L}$ such that $L$ is isometric to $\tilde{L}/\Gamma$ and the following hold:\begin{itemize}  \item $\tilde{L}$ contains the horoball determined by some $\bu\in U$, where $U$ is the common light-like eigenspace of $\Gamma_1$;
\item  $\tilde{L}$ is contained in the horoball determined by some $\bu'\in U$; and
\item $\Gamma_1$ acts cocompactly on the horosphere determined by any $\bu\in U$.\end{itemize}
Conversely, any parabolic subgroup $\Gamma_1<\Gamma$ stabilizes a component $\widetilde{L}$ of $\pi^{-1}(L)$ for some cusp $L$ of $M$, and any $\bu$ in the fixed point of $\Gamma_1$ determines a horoball containing some such $\widetilde{L}$.\end{lemma}

\begin{proof}  This is Case 2 in the proof of Theorem D.3.3 of \cite{BenPet}.  In particular, $\Gamma_1$ is parabolic by construction.  It and $\tilde{L}$ are defined at the beginning of that case (pp.~147--148), and $L = \pi(\tilde{L})$ is isometric to $\tilde{L}/\Gamma_1$ by Lemma D.3.7.

The proof of \cite[Th.~D.3.3]{BenPet} takes place in the upper half-space model $\mathbb{R}^{n-1}\times(0,\infty)$ for $\mathbb{H}^n$, with $\Gamma_1$ consisting of isometries fixing the point $\infty$ of $\partial \mathbb{H}^n = (\mathbb{R}^{n-1}\times\{0\})\cup\{\infty\}$.  Each such is of the form $(y,t)\mapsto (I(y),t)$ for a Euclidean isometry $I$ of $\mathbb{R}^{n-1}$ (see point 3 on p.~142).  It preserves each horosphere centered at $\infty$, of the form $\mathbb{R}^{n-1}\times\{t\}$ for some $t>0$, and acts as an isometry of the inherited Euclidean metric.

In the hyperboloid model the role of $\infty$ is filled by $p(U\cap L^+)$ for the common one-dimensional light-like eigenspace $U$ of elements of $\Gamma_1$, and horospheres centered at $p(U\cap L^+)$ are of the form $\{\bx\circ\bu=-1\}\cap \mathbb{H}^n$ for $\bu\in U\cap L^+$ (cf.~Definition \ref{at infinity}).

The proof of Lemma D.3.8 identifies $\tilde{L}$ as the set of points above the graph of a certain continuous function $Q\co\mathbb{R}^{n+1}\to (0,\infty)$.  It is implicit in the proof (and easy to verify) that $Q$ is $\Gamma_1$-invariant in the obvious way: namely, that $Q(y) = Q(I_{\gamma}(y))$ for any $y\in\mathbb{R}^{n-1}$ and $\gamma\in\Gamma_1$, where $I_{\gamma}$ is the Euclidean isometry such that $\gamma(y,t) = (I_{\gamma}(y),t)$.

It is also true that $V = \mathbb{R}^{n-1}/\Gamma_1$ is compact --- this the final line of Theorem D.3.3 and the final line of its proof --- whence $Q$ attains a minimum $t_0$ and maximum $t_1$ on $\mathbb{R}^{n-1}$.  Hence we have $\mathbb{R}^{n-1}\times [t_1,\infty)\subset \tilde{L}\subset \mathbb{R}^{n-1}\times[t_0,\infty)$, giving the first two bullet points.  That $V$ is compact also directly implies the final point.

Now let $\Gamma_1$ be an arbitrary parabolic subgroup of $\Gamma$ with shared light-like eigenspace $U$.  Any fixed $g\in\Gamma_1-\{\mathit{id}\}$ preserves each $\bu\in U$ and, for each such $\bu$ with $u_0>0$, the horosphere $S_{\bu}=\{\bx\circ\bu=-1\}\cap\mathbb{H}^n$ that it determines.  For fixed such $\bu$ and any $\bx\in S_{\bu}$, let $\gamma_{\bx}$ be the geodesic ray of Lemma \ref{convex component}.  

One checks directly that $d_H(\gamma(t),g.\gamma(t))$ decreases in $t$, approaching $0$ as $t\to\infty$, so there exists $t_0\geq0$ such that $\pi$ projects $\gamma_{\bx}[t_0,\infty)$ into a component $L$ of $M_{(0,\epsilon]}$.  $L$ is a cusp since $\pi(\gamma_{\bx}[t_0,\infty))$ contains points with arbitrarily small injectivity radius, so a component $\widetilde{L}$ of $\pi^{-1}(L)$ contains $\gamma_{\bx}[t_0,\infty)$.  That $\widetilde{L}$ is stabilized by $\Gamma_1$ now follows from Lemma \ref{intersect balls} and the above, since the horoball containing $\widetilde{L}$ has non-compact intersection (containing $\gamma_{\bx}[t_0,\infty)$) with the horoball determined by $\bu$.

For any $\epsilon'<\epsilon$ there is a component $L'$ of $M_{(0,\epsilon']}$ contained in $L$, and a component $\widetilde{L}'$ of $\pi^{-1}(L')$ contained in $\widetilde{L}$.  It follows from the definition of $Q$ that $\widetilde{L}'$ the set of points above the graph of $Q+k$ for some $k = k(\epsilon')$ approaching $\infty$ as $\epsilon'>0$.  The horoball $\{\bx\circ\bu\geq -1\}\cap\mathbb{H}^n$ bounded by $S_{\bu}$ therefore contains $\widetilde{L'}$ for some such $\epsilon'$.\end{proof}

\subsection{Back to the Epstein--Penner construction}  A key idea in \cite{EpstePen} is that rotating a support plane around an codimension-one subplane produces a new one in certain circumstances (see eg.~the paragraph spanning pp.~74--75 there).  This is also quite useful in the current setting, but more delicate when the support plane in question is parallel to a light-like subspace: such planes have non-compact intersection with $\mathbb{H}^n$.  

\begin{lemma}\label{rotate plane}  Suppose $\cals\subset\mathbb{H}^n$ is locally finite and $P$ is a support plane for $\mathrm{Hull}(\cals)$, parallel to a light-like subspace $V$ of $\mathbb{R}^{n+1}$, that separates $\bo$ from $\cals$.  If $Q\subset P$ is an $(n-1)$-dimensional affine plane containing no translate of $V^{\perp}$, such that the half-space $P^+$ of $P$ bounded by $Q$ that has $P^+\cap U_{n+1}$ non-compact contains $P\cap\cals$, then rotating around $Q$ ``in the space-like direction'' (see below) produces a family of space-like planes separating $\cals$ from $\bo$ and $P-P^+$.\end{lemma} 

\begin{proof}  We will assume $Q\cap\mathbb{H}^n$ is non-empty.  If it were empty then for the minimal $t_0>0$ such that $(Q+t_0\bu)\cap\mathbb{H}^n\neq\emptyset$, $Q+t_0\bu$ would still satisfy the hypotheses.  The plane produced by rotating $P$ about $Q+t_0\bu$ in the sense below would produce a plane that separates $\cals$ from the plane produced by rotating $P$ about $Q$ by the same amount.

Thus fix $\bx_0\in Q\cap\mathbb{H}^n$ and let $V_0 = Q-\bx_0$.  This $(n-1)$-dimensional subspace of $V$ is  space-like since it does not contain $V^{\perp}$.  Let $\bu\in V^{\perp}$ satisfy $\bu\circ\bx_0 = -1$.  The Lorentz-orthogonal complement to $V_0$ in $\mathbb{R}^{n+1}$ also contains some $\bw\in\mathbb{H}^n$, hence is spanned by $\bu$ and $\bw$.

For $t\in\mathbb{R}$ let $V_t$ be the vector space spanned by $V_0$ and $\bu-t\bw$, and let $P_t = V_t+\bx_0$.  If $t>0$ we say $P_t$ is obtained by rotating $P$ around $Q$ \textit{in the space-like direction}.  The following hold:\begin{itemize}
\item  For any $t\neq 0$, $V\cap V_t = V_0$ and $P\cap P_t = Q$.
\item  For $0<t<-\bu\circ\bw$, $V_t$ is space-like, with time-like normal vector $\bn_t=\bu+s\bw$ for $s = t\bu\circ\bw/(t+\bu\circ\bw)$.  (Note that $\bu\circ\bw<0$ by (\ref{x circ y}).)
\item  For $0<t<-\bu\circ\bw$, the half-space $B_t^-=\{\bx\,|\,\bx\circ\bn_t\geq\bx_0\circ\bn_t\}$ bounded by $P_t$ that contains $\bo$ has convex intersection with $\mathbb{H}^n$ (by Lemma \ref{convex component}, since $\bo\circ\bn_t = 0$).
\item  The half-space $B^-$ bounded by $P$ and containing $\bo$ contains $P_t^+ = \{a\bu_t+\bq\,|\, a>0\ \mbox{and}\ \bq\in Q\}$, where $\bu_t=\bu-t\bw$. (One checks directly that $\bu\circ \bx\geq -1$ for $\bx\in P_t^+$.)
\item  $B_t^-$ contains $P^-=\{a\bu+\bq\,|\,a>0, \bq\in Q\}$. (As above, $\bn_t\circ\bx\geq\bn_t\circ\bx_0$ for $\bx\in P^-$.)\end{itemize}

Let $B^+$ and $B_t^+$ be the half-spaces opposite $B^-$ and $B_t^-$, respectively.  Since $\bo$ and $P^-$ are in $B_t^-$, the goal is to show for small $t>0$ that $\cals\subset B_t^+$.  The claim below will be a key aid:

\begin{claim}\label{where's the max} The set of initial entries of points of $P^-\cap U_{n+1}$ and the set of initial entries of points of $P_t^-\cap U_{n+1}$ share a maximum $M$, attained in $Q\cap\mathbb{H}^n$.\end{claim}

The case of $P^-\cap U_{n+1}$ is a warm-up.  Recall from Lemma \ref{Euclid's planes} that $P^-\cap U_{n+1}$ is compact, so the first entry function attains a maximum.  Since $P$ is affine and not parallel to $\{0\}\times\mathbb{R}^n$ the maximal first entry of $P^-\cap U_{n+1}$ does not occur in its interior.  It is thus attained in the boundary $(Q\cap U_{n+1})\cup(P^-\cap\mathbb{H}^n)$.  Since $Q$ is affine the maximal first entry of $Q\cap U_{n+1}$ occurs in its boundary $Q\cap\mathbb{H}^n\subset P\cap\mathbb{H}^n$.  So it suffices to consider $P^-\cap\mathbb{H}^n$.

We will use the method of Lagrange multipliers to show that the first entry function has no local maximum on $P\cap\mathbb{H}^n$, whence the maximal first entry of $P^-\cap\mathbb{H}^n$ must be attained in its boundary $Q\cap\mathbb{H}^n$.  This will prove the claim for $P^-$.

For $\bx = (x_0,\hdots,x_n)\in\mathbb{H}^n$, $x_0 = \sqrt{1+x_1^2+\hdots+x_n^2}$.  We find local extrema of this function subject to the constraint that $\bx$ is in $P$; i.e. $\bx\circ\bu = -1$.  Write $\bu = (u_0,\hdots,u_n)$.  After simplifying we find that at a local extremum for the first entry function there exists $\lambda$ with:
$$ (x_1,\hdots,x_n) = \lambda\left(u_1\sqrt{1+x_1^2+\hdots+x_n^2} - u_0x_1,\hdots,u_n\sqrt{1+x_1^2+\hdots+x_n^2}-u_0x_n\right) $$
This implies that $x_i/u_i = x_j/u_j$ for any $i,j>0$, so $(x_1,\hdots,x_n) = \gamma(u_1,\hdots,u_n)$ for some scalar $\gamma$.  Applying the constraint equation at a local extremum thus yields:\begin{align*}
  -u_0\sqrt{1+\gamma^2(u_1^2+\hdots+\hdots u_n^2)} + \gamma(u_1^2+\hdots+u_n^2) = -1 \end{align*}
Since $\bu$ is light-like we have $u_0^2 = u_1^2+\hdots+u_n^2$.  Substituting above and simplifying produces a linear equation in $\gamma$ with the single solution $\gamma = (u_0^2-1)/2u_0^2$.  It follows that there is a unique local extremum for the first entry function on $P\cap\mathbb{H}^n$.  Since first entries attain a global minimum this is it, and there is no local maximum.  The claim follows for $P^-$.

We now modify the argument to treat $P_t^-$.  Since $V_t$ is space-like $P_t\cap U_{n+1}$ is compact (Lemma \ref{convex component}), so there is a maximal first entry in $P_t^-\cap U_{n+1}$.  For small enough $t>0$, $P_t$ is again not parallel to $\{0\}\times\mathbb{R}^n$, so the maximum is in the boundary $(Q\cap\mathbb{H}^n)\cup(P_t^-\cap\mathbb{H}^n)$.  As before it suffices to consider $P_t^-\cap\mathbb{H}^n$.

Lagrange multipliers will show in this case that the first entry function on $P_t\cap\mathbb{H}^n$ has exactly two local extrema, one of which must therefore be the maximum and the other the minimum (attained by compactness).  We will show that the unique local maximum lies in $B^-$ and hence, by the fourth bullet above, in $P_t^+$, and it will follow that the maximal first entry on $P_t^-\cap\mathbb{H}^n$ lies in its boundary $Q\cap\mathbb{H}^n$.

We find local extrema of $x_0$ for $\bx=(x_0,\hdots,x_n)$ as before, but this time with the constraint $\bx\circ\bn_t = k\doteq\bx_0\circ\bn_t$ (i.e.~$\bx\in P_t$).  Taking $\bn_t = (n_0,\bn_0)$, as before we find that extrema occur at $\bx$ with $(x_1,\hdots,x_n) = \gamma\bn_0$ for some scalar $\gamma$.  At such an $\bx$, $x_0 = \sqrt{1+\gamma^2\|\bn_0\|^2}$ and the constraint equation takes the form:\begin{align}\label{constraint II}
 -n_0\sqrt{1+\gamma^2\|\bn_0\|^2} + \gamma\|\bn_0\|^2 = k \end{align}
This determines a quadratic in $\gamma$, with leading coefficient $\|\bn_0\|^2(\|\bn_0\|^2-n_0^2) = \|\bn_0\|^2\bn_t\circ\bn_t$.  Since $\bn_t$ is time-like this coefficient is non-zero.  The discriminant is $\bn_t\circ\bn_t+k^2$.  By (\ref{x circ y}), $k<-\sqrt{-\bn_t\circ\bn_t}$ (recall $\bx_0\in\mathbb{H}^n$), so the discriminant is positive and the quadratic has two solutions.  These correspond to the minimal and maximal first entries on $P_t\cap\mathbb{H}^n$.  The maximum occurs at the larger possibility for $\gamma$: $(k\|\bn_0\|-n_0\sqrt{\bn_t\circ\bn_t+k^2})/(\|\bn_0\|\bn_t\circ\bn_t)$.

Using the above, the Lorentzian inner product of the maximum point $\bx$ and $\bu$ is:\begin{align}\label{check convex}
 \bx\circ\bu 
   & = -x_0u_0 + \gamma\bn_0\cdot\bu_0 = (\gamma n_0 - x_0)u_0 + \gamma\bn_t\circ\bu\nonumber \\
   & = \frac{\sqrt{\bn_t\circ\bn_t+k^2}}{\|\bn_0\|}\left(u_0-n_0\frac{\bn_t\circ\bu}{\bn_t\circ\bn_t}\right) + k\frac{\bn_t\circ\bu}{\bn_t\circ\bn_t} \end{align}
Here we have used (\ref{constraint II}) to rewrite $x_0$ as $(\gamma\|\bn_0\|^2-k)/n_0$, and simplified.  An explicit computation shows that $\bn_t\circ\bu/\bn_t\circ\bn_t$ approaches $1/2$ as $t\to 0$.  That $\bn_t\to\bu$, implies that $\bn_t\circ\bn_t \to 0$, $n_0\to u_0$, $\bn_0\to\bu_0$, and $k\to -1$ as $t\to 0$.  It follows that $\bx\circ\bu\to 0$ as $t\to 0$.  In particular, for small $t>0$ the maximum point is in $B^-$.  The claim follows for $P_t^-$.

\begin{claim}\label{upper half}  $C\doteq\{\bx\in B^+\cap U_{n+1}\,|\,x_0\geq M\}$ is contained in $B_t^+$.\end{claim}

$B^+$ is divided into two convex ``half-spaces'' by $P_t$, each the closure of a component of $B^+ - P_t$.  By the fourth bullet above, $P_t \cap B^+ = P_t^-$ and $P^-\subset B_t^-$.  Since $B_t^-\cap U_{n+1}$ is compact (Lemma \ref{convex component}, since $P_t$ is space-like), the half-space of $B^+$ bounded by $P_t^-$ and $P^-$ has compact intersection $K$ with $U_{n+1}$. For any $\bx$ in $K$, the ray $\bx+(t,0,\hdots,0)$ remains in $U_{n+1}$ for all $t\geq 0$ but exits $K$ at some $t_0>0$.  This point is in $P_t^-\cap U_{n+1}$ or $P^-\cap U_{n+1}$, so by claim \ref{where's the max} its first coordinate (hence also that of $\bx$) is at most $M$.  The claim follows.

Since $\cals\subset B^+\cap U_{n+1}$, claim \ref{upper half} implies that $B_t^+$ contains $\{\bs\in\cals\,|\,s_0\geq M\}$.  We now show explicitly that the set $\cals_0$ of $\bs\in\cals$ with $s_0 \leq M$ is contained in $B_t^+$.  Since $\cals$ is locally finite $\cals_0$ is finite, so there is some $\epsilon>0$ such that $\bs\circ\bu < -1-\epsilon$ for all $\bs\in\cals_0-Q$.  (Recall that $\cals\cap P^-\subset Q$.) For $t$ near enough to $\bo$ it follows that $\bs\circ\bn_t < -1-\epsilon/2$ for all such $\bs$.  (Note that $s$ as defined in the second bullet above approaches $0$ as $t\to 0$.)  On the other hand, $t$ can also be chosen near enough to $0$ that $\bx_0\circ\bn_t > -1-\epsilon/2$.  Thus $\cals_0\subset B_t^+$ for small $t>0$.  This proves the lemma.\end{proof}

\begin{corollary}\label{gots a point}Suppose $\cals\subset\mathbb{H}^n$ is locally finite.  If $P$ is a support plane for $\mathrm{Hull}(\cals)$ parallel to a light-like subspace of $\mathbb{R}^{n+1}$ that separates $\bo$ from $\cals$, then $P\cap\cals\neq\emptyset$.\end{corollary}

\begin{proof}  Supposing there exists such a support plane $P$ with $P\cap\cals = \emptyset$, for arbitrary $h > 1$ consider $Q_h = P\cap (\{h\}\times\mathbb{R}^n)$.  Since $\{h\}\times\mathbb{R}^n$ is space-like the half-space $P^-$ bounded by $Q_h$ as in Lemma \ref{Euclid's planes} has compact intersection with $U_{n+1}$ (it is contained in the half-space bounded by $\{h\}\times\mathbb{R}^n$ that has compact intersection with $\mathbb{H}^n$, cf.~Lemma \ref{convex component}).  Therefore by Lemma \ref{rotate plane} rotating $P$ a small amount around $Q_h$ in the space-like direction produces a plane separating all points of $P^-$ from $\cals$; hence also from $\mathrm{Hull}(\cals)$.  But since $h$ was arbitrary it follows that $P\cap\mathrm{Hull}(\cals)=\emptyset$, a contradiction.\end{proof}

\begin{proposition}\label{good support}  Suppose $\cals$ is locally finite and invariant under a torsion-free lattice $\Gamma$ of $\mathit{SO}^+(1,n)$.  If $P=V+\bx_0$ is a support plane for $\mathrm{Hull}(\cals)$ that separates $\bo$ from $\cals$, with $V$ light-like, then a parabolic subgroup of $\Gamma$ preserves $V^{\perp}$ and acts co-compactly on $P\cap\mathbb{H}^n$.\end{proposition}

\begin{proof}  Suppose $P = V+\bx_0$ is a support plane for $\mathrm{Hull}(\cals)$ with $V$ light-like, and let $B$ be the half-space bounded by $P$ such that $B\cap \mathbb{H}^n$ is convex.  Recall from Lemma \ref{convex component} that $B = \{\bx\,|\,\bx\circ\bu\geq -1\}$, where $\bu\in V^{\perp}$ satisfies $\bu\circ\bx=-1$ for all $\bx\in P$.  In particular, $\bo\in B$.  Since $P$ separates $\bo$ from $\cals$, $\cals\cap B\subset P$.

For each $k$ between $-1$ and $0$ let $P_k = \{\bx\circ\bu = k\}$, and let $B_k = \{\bx\circ\bu \geq k\}$ be the half-space bounded by $P_k$ that is contained in $B$.  Let us note a few things about the $P_k$.  For each $k$, $P_k\cap\mathbb{H}^n$ is a horosphere, since taking $\bu_k = \frac{-1}{k}\bu\in V^{\perp}$ we may write $P_k = \{\bx\circ\bu_k = -1\}$.  The collection $\{P_k\cap\mathbb{H}^n\,|\,-1\leq k < 0\}$ foliates $B\cap\mathbb{H}^n$, since it follows from (\ref{x circ y}) that $\bx\circ\bu < 0$ for any $\bx\in\mathbb{H}^n$.  Finally, for $S_k = P_k\cap\mathbb{H}^n$ Lemma \ref{bad balls} asserts that $d_H(S_k,S) = \ln (-1/k)$.

Let $M = \mathbb{H}^n/\Gamma$ and $\pi\co\mathbb{H}^n\to M$ the quotient map.  Since all points of $\cals$ lie in $S$ or outside $B$, the last fact above implies that they all have distance at least $\ln (-1/k)$ from all points of $B_k$.  This in turn implies the same inequality on the distance in $M$ between $\pi(\cals)$ and $\pi(B_k\cap\mathbb{H}^n)$.  (The distance between $x,y\in M$ is the minimum, for any fixed $x_0\in\pi^{-1}(x)$, of the hyperbolic distances from $x_0$ to the points of $\pi^{-1}(y)$.)

As $\epsilon\to 0$, the thick parts $M_{[\epsilon,\infty)}$ form an exhaustion of $M$ by compact subsets, so for fixed $\bs\in \cals$ there is an $\epsilon >0$ such that $\pi(\bs)\in M_{[\epsilon,\infty)}$.  For $k<0$ near enough to $0$ that $\ln (-1/k)$ is larger than the diameter of $M_{[\epsilon_0,\infty)}$, it follows that $\pi(B_k\cap\mathbb{H}^n)$ is disjoint from $M_{[\epsilon_0,\infty)}$ and hence contained in a component $L$ of the $\epsilon$-thin part.  As $\pi(B_k\cap\mathbb{H}^n)$ contains points at arbitrarily large distance from $\pi(\bs)$ it is non-compact, so it lies in a cusp.

Let $\tilde{L}$ be the component of $\pi^{-1}(L)$ containing $B$, and let $\Gamma_1$ be the parabolic subgroup supplied by Lemma \ref{invariant horoballs} with light-like eigenspace of $\Gamma_1$.  By Lemma \ref{invariant horoballs}, $\tilde{L}$, and hence also $B\cap\mathbb{H}^n$, is contained in the horoball determined by some $\bu\in U$.  This implies that $U=V^{\perp}$, since by Lemma \ref{intersect balls} horoballs determined by linearly independent light-like vectors have compact intersection.  Lemma \ref{invariant horoballs} thus also implies that $\Gamma_1$ acts cocompactly on $P\cap\mathbb{H}^n$.\end{proof}

\begin{proposition}\label{horospherical cells}  Suppose $\cals$ is locally finite and invariant under a torsion-free lattice $\Gamma$ of $\mathit{SO}^+(1,n)$.  If $P$ is a support plane for $\mathrm{Hull}(\cals)$ parallel to a light-like subspace of $\mathbb{R}^{n+1}$ then $P\cap\mathrm{Hull}(\cals)$ is an $n$-dimensional face $F$ of $\mathrm{Hull}(\cals)$ with $F = \mathrm{Hull}(P\cap\cals)$, and:\begin{itemize}
\item  the stabilizer $\Gamma_U$ of $U = V^{\perp}$ in $\Gamma$ is parabolic, and $F$ and $r_n(F)$ are $\Gamma_U$-invariant;
\item  each face of $F$ is a compact face of $\mathrm{Hull}(\cals)$ of the form $P'\cap\mathrm{Hull}(\cals)=\mathrm{Hull}(P'\cap\cals)$, where $P'$ is a support plane for $\mathrm{Hull}(\cals)$ parallel to a space-like subspace of $\mathbb{R}^{n+1}$; and
\item the collection of faces of $F$ is locally finite and finite up to the $\Gamma_U$-action, and $\partial F$ is the union of the $(n-1)$-dimensional faces.\end{itemize}\end{proposition}

\begin{proof}  By Proposition \ref{good support} there is a parabolic subgroup of $\Gamma$ preserving $P$ and acting cocompactly on $P\cap\mathbb{H}^n$.  The maximal such subgroup $\Gamma_U$ is the stabilizer of $U$ in $\Gamma$, since all elements fixing $U$ are parabolic (cf.~\cite[Lemma D.3.6]{BenPet}).  By Lemma \ref{gots a point}, $\cals_U = P\cap \cals$ is non-empty, and since $\cals$ is $\Gamma$-invariant $\cals_U$ is $\Gamma_U$-invariant.

Let $\bu\in V^{\perp}$ satisfy $\bu\circ\bx_0=-1$.  For any support plane $Q$ for $\mathrm{Hull}(\cals_U)$ in $P$, the half-space $P^-\doteq \{\bq-a\bu\,|\,\bq\in Q\ \mbox{and}\ a\geq 0\}$ of $P$ bounded by $Q$ has compact intersection with $U_{n+1}$, by Lemma \ref{Euclid's planes}.  For if not then $Q\cap\mathbb{H}^n$ would be a Euclidean hyperplane of $S=P\cap\mathbb{H}^n$, thus bounding a half-space containing points of $S$ arbitrarily far from $\cals_U$.  But this would contradict cocompactness of the $\Gamma_U$-action.  

Cocompactness also implies that $\cals_U$ is contained in the half-space $P^+$ opposite $P^-$, and that it is not entirely contained in $Q$.  By Lemma \ref{rotate plane}, $P$ can be rotated in the space-like direction around $Q$ to produce a new support plane $P'$ for $\mathrm{Hull}(\cals)$ that separates $P^+$ from $\bo$, excluding all points of $P - P^+$ from $\mathrm{Hull}(\cals)$.  It follows that $F$, which could \textit{a priori} contain $\mathrm{Hull}(\cals_U)$ properly, does not, and also that $F = \mathrm{Hull}(\cals)$ is $n$-dimensional.

$P'$ is parallel to a space-like subspace and satisfies $P'\cap\mathrm{Hull}(\cals) = Q\cap\mathrm{Hull}(\cals) = F_0$, where $F_0$ is the face of $F$ contained in $Q$.  Therefore $F_0=\mathrm{Hull}(P'\cap\cals)$ is compact and equal to $\mathrm{Hull}(P'\cap\cals)$ by Lemma \ref{rotate space}, and we have the second bullet above.

An argument of Epstein--Penner, from the first full paragraph on \cite[p.~75]{EpstePen}, shows local finiteness of the collection of faces.  For a compact set $K\subset P$ and a sequence of distinct faces $F_1,F_2,\hdots$ of $F$ intersecting $K$, the corresponding sequence of support planes $Q_1,Q_2,\hdots$ in $P$ has a subsequence that converges (in the sense of Lemma \ref{plane convergence}) to a support plane $Q_0$ for $F$.  By the above $Q_0\cap U_{n+1}$ is compact, so by Fact \ref{small neighborhood} some fixed neighborhood of $Q_0\cap U_{n+1}$ contains $Q_i\cap U_{n+1}$ for all large enough $i$.  But $P\cap\cals$ is locally finite so this neighborhood contains only finitely many of its points.  The $F_i$ are thus not all distinct (recall Lemma \ref{rotate space}).

That every $\partial F$ is the union of $(n-1)$-dimensional faces also follows from an argument of \cite{EpstePen}, in the paragraph spanning pp.~74--75 there.  If $Q$ is a support plane for $F$ in $P$ with $Q\cap F\subset Q_0$ for some $(n-2)$-dimensional subspace $Q_0$ then since $Q\cap U_{n+1}$ is compact, by Fact \ref{small neighborhood} rotating $Q$ about $Q_0$ by a small amount produces new support planes for $F$.  Since a limit of support planes is a support plane (by Lemma \ref{plane convergence}), there is a closed interval about $0$ of rotations through support planes, with boundary points consisting of support planes containing points of $\cals$ outside $Q_0$.  Such a plane intersects $F$ in a face properly containing $Q_0\cap F$, and the assertion follows.

Since each face of $F$ is the convex hull of its intersection with $\cals$, $\cals_U$ is the set of $0$-dimensional faces.  There are only finitely many $\Gamma_U$-orbits in $\cals_U$ by cocompactness, and each point of $\cals_U$ is in only finitely many faces of $F$ by local finiteness.  Since each face of $F$ contains a vertex, it follows that there are only finitely many $\Gamma_U$-orbits of faces.\end{proof}

\begin{corollary}\label{its a poly}  Suppose $\cals$ is locally finite and invariant under a torsion-free lattice $\Gamma$ of $\mathit{SO}^+(1,n)$.  If $P$ is a support plane for $\mathrm{Hull}(\cals)$ parallel to a light-like vector subspace $V$ of $\mathbb{R}^{n+1}$ then for $F = P\cap\mathrm{Hull}(\cals)$, $r_n(F)$ is an $n$-dimensional convex polyhedron equal to the closed convex hull of $P\cap\cals$ in $\mathbb{H}^n$ and containing the horoball determined by some $\bu\in V^{\perp}$.\end{corollary}

\begin{proof}  That $r_n(F)$ is the closed convex hull of $P\cap\cals$ will follow from Lemma \ref{convex hull} upon showing that $r_n(F)$ is closed.  This in turn is a consequence of the following fact: there is a sequence of $\{V_k\}$ time-like subspaces of $\mathbb{R}^{n+1}$, each bounding a half-space $B_k$ with $B_k\cap P$ compact, such that $B_1\cap P \subset B_2\cap P\subset\hdots$ exhausts $P$ and $d_H(V_k\cap P,V_{k+1}\cap P) = 1$ for each $k$.  To obtain the $V_k$ fix $\bx\in P\cap\mathbb{H}^n$, and for $\gamma_{\bx}$ as in Lemma \ref{convex component} let $V_k = \gamma_{\bx}'(k)^{\perp}$.

For each $k$, $r_n$ preserves $V_k$ and takes $B_k\cap P$ to $B_k\cap (B\cap\mathbb{H}^n)$, where $B$ is the half-space bounded by $P$ containing $\bo$ (recall Lemma \ref{ratatouille}).  By the paragraph above, a convergent sequence in $r_n(F)$ is entirely contained in some $B_k\cap(B\cap\mathbb{H}^n)$.  Since $B_k\cap F$ is compact its image under $r_n$ is closed, hence contains the limit point.

Enumerate the set of faces of $F$ as $F_1,F_2,\hdots$ and let $Q_1,Q_2,\hdots$ be the corresponding sequence of support planes for $F$ in $Q$, with $F_i = Q_i\cap F$ for each $i$.  For each $i$ let $V_i$ be the time-like vector subspace of $\mathbb{R}^{n+1}$ spanned by $Q_i$.  We claim that $\{V_i\cap\mathbb{H}^n\}$ is locally finite in $\mathbb{H}^n$.

To prove the claim, suppose not and let $\bx_0$ be an accumulation point for the $V_i\cap\mathbb{H}^n$.  Since there are finitely many $\Gamma_U$-orbits of the $F_i$ (Proposition \ref{horospherical cells}), passing to a subsequence yields $\bx_i\to \bx_0$ with $\bx_i\in V_i = g_i.V_1$, where $g_i\in\Gamma_U$ for each $i$.  If $V$ is the light-like subspace parallel to $P$ and $\bu\in V^{\perp}$ has the property that $\bu\circ\bx = -1$ for each $\bx\in P$, then $\bu\circ\bx_0-1 \leq \bu\circ\bx_i$ for all large enough $i$.  We may substitute $\bu\circ g_i^{-1}\bx_i = \bu\circ\bx_i$ in this inequality since $g_i\in\mathit{SO}^+(1,n)$ fixes $\bu$, and note that thus all $g_i^{-1}\bx_i$ lie in the horoball $\{\bx\circ\bu\geq \bx_0\circ\bu-1\}$.

For any fixed $\by\in V_1\cap\mathbb{H}^n$, $d_H(\bx,g_i\by)\to\infty$ as $i\to\infty$ since $\Gamma$ acts discontinuously on $\mathbb{H}^n$.  Since $\bx_i\to\bx$ the same holds for $d_H(\bx_i,g_i\by) = d_H(g_i^{-1}\bx_i,\by)$.  But the $g_i^{-1}\bx_i$ all lie in the compact subset $\{\bu\circ\bx_0-1\leq\bu\circ\bx\}$ of $V_1\cap\mathbb{H}^n$ (cf.~Lemma \ref{intersect ball}), a contradiction.

The claim is proved, so to show $r_n(F)$ is a polyhedron we observe that $r_n(F) = \left(\bigcap B_i\right)\cap\mathbb{H}^n$, where $B_i$ is the half-space bounded by $V_i$ and containing $F$ for each $i$.  Since $r_n$ preserves each $V_i$ and $B_i$ it is clear that $r_n(F)\subset\left(\bigcap B_i\right)\cap\mathbb{H}^n$.  Recall from Proposition \ref{horospherical cells} that $\partial F = \bigcup F_i$, so since $r_n(F)$ is closed and $r_n|_{P\cap U_{n+1}}$ is a local homeomorphism, $\partial r_n(F) = \bigcup r_n(F_i)$.  This implies equality: for fixed $\bx\in \mathit{int}\,r_n(F)$ and any $\by\in\mathbb{H}^n-F$ the geodesic arc from $\bx$ to $\by$ exits $r_n(F)$ in some $F_i$; hence also exits $\left(\bigcap B_i\right)\cap\mathbb{H}^n$ there.

If $B = \{\bx\circ\bu\geq -1\}$ is the half-space bounded by $P$ that intersects $\mathbb{H}^n$ in the horoball containing $r_n(F)$, Lemma \ref{intersect ball} implies that $(V_i\cap B)\cap\mathbb{H}^n$ is compact for any $i$, so $\bx\circ\bu$ attains a non-zero maximum on it.  This is invariant under the $\Gamma_U$-action, so since there are finitely many $\Gamma_U$-orbits of the $F_i$ there exists $k<0$ such that $\bx\circ\bu\leq k$ for all $\bx\in F_i$ and any $i$.  Therefore $r_n(F)$ contains the horoball $B_k\cap\mathbb{H}^n$, where $B_k = \{\bx\circ\bu\geq k\}$.\end{proof}

\begin{lemma}\label{locally finite}  If $\cals$ is locally finite and invariant under a torsion-free lattice $\Gamma$ of $\mathit{SO}^+(1,n)$ then the collection of faces of $\mathrm{Hull}(\cals)$ is locally finite.\end{lemma}

\begin{proof}  Given a compact set $K\subset U_{n+1}$ that intersects infinitely many distinct visible faces $F_n$, the corresponding sequence $\{P_n\}$ of support planes has a subsequence that converges (in the sense of Lemma \ref{plane convergence}) to a support plane $P$.  If the parallel subspace $V$ to $P$ is space-like then $P\cap U_{n+1}$ is compact and we follow Epstein--Penner \cite[p.~75]{EpstePen}; see the proof of Proposition \ref{horospherical cells}.  We therefore suppose below that $V$ is light-like.

For each $n$ let $\bx_n\in F_n$, and let $\eta_n$ be a unit-length Euclidean outward normal to the half-space $B_n^-$ bounded by $P_n$ and containing $\bo$.  After subsequencing we may take $\{(\bx_n,\eta_n)\}\to (\bx_0,\eta)$ for some $\bx_0$ in the face $F = P\cap\mathrm{Hull}(\cals)$ and Euclidean outward normal $\eta$ to the half-space $B^-$ bounded by $P$ that contains $\bo$.

By Remark \ref{Euclid vs Lorentz}, $\bar{\eta}\circ\bx_0 < 0$, so there is a positive scalar multiple $\bu$ of $\bar{\eta}$ with $\bu\circ\bx_0 = -1$.  We further have $\bu\in V^{\perp}$ and $B^- = \{\bx\circ\bu\geq-1\}$.  Scaling the $\bar{\eta}_n$ by the same factor produces a sequence $\{\bu_n\}\to\bu$ such that $\bu_n$ is Lorentz-orthogonal to the subspace $V_n$ parallel to $P_n$ for each $n$, and $B_n^- = \{\bx\circ\bu_n\geq\bx_n\circ\bu_n\}$.

By Proposition \ref{horospherical cells} the collection of faces of $F$ is locally finite, so after excluding finitely many $F_n$ we may assume the support planes $P_n$ satisfy $P_n\neq P$ for all $n$.  $P_n$ is also not parallel to $P$, being a support plane, so $Q_n = P_n\cap P$ is non-empty for all $n$.

By Proposition \ref{good support}, a parabolic subgroup $\Gamma_U$ of $\Gamma$ preserves $V^{\perp}$ and acts cocompactly on $P\cap\mathbb{H}^n$.  Since $P\cap\cals$ is non-empty (Corollary \ref{gots a point}) and $\Gamma_U$-invariant, cocompactness implies there exists $J>0$ such that all points of $P\cap\mathbb{H}^n$ are within $J$ of $P\cap\cals$.  It therefore follows from Lemma \ref{Euclid's planes} that none of the $Q_n$ contain a translate of $\bu$, since otherwise $B_n^-\cap P$ would contain points arbitrarily far from $P\cap\cals$.

If an infinite subsequence of the $P_n$ were parallel to light-like subspaces then Lemma \ref{intersect balls} would imply that the corresponding subsequence of the $(B_n^-\cap P)\cap \mathbb{H}^n$ contains empty balls of arbitrarily large diameter, again contradicting cocompactness.  Upon discarding a finite collection of the $P_n$ we may thus assume that each parallel subspace $V_n$ is space-like (also recall Lemma \ref{no time}), so $B_n^-\cap U_{n+1}$ is compact (Lemma \ref{convex component}).

Since $\{\bx_n\in F_n\}\to \bx_0\in F$, one can show by a Euclidean argument that there is a sequence $\{\by_n\in Q_n\}$ converging to $\bx_0$.  It follows that a subsequence of the $Q_n$ converges to a support plane $Q\subset P$ for $F$ containing $\bx_0$.  
As was showed in the proof of Proposition \ref{horospherical cells}, $Q\cap U_{n+1}$ is compact.
Let $M$ be the maximum of initial entries of points of $Q\cap U_{n+1}$.  Since $Q_n\to Q$, by Fact \ref{small neighborhood} the first entries of points of $Q_n\cap U_{n+1}$ are at most $M+1$ for large enough $n$.  

\begin{claim}\label{where's the maxx}  Let $B^+$ be the half-space bounded by $P$ opposite $B^-$, and for each $n$ let $P_n^- = P_n\cap B^+$.  The maximal first entry of points of $P_n^-\cap U_{n+1}$ is attained in $Q_n\cap\mathbb{H}^n$.\end{claim}

This is analogous to claim \ref{where's the max} of Lemma \ref{rotate plane}, and its proof is analogous to the $P_t^-$ case there.  In fact upon replacing $P_t$ and $Q$ with $P_n$ and $Q_n$, respectively, and $\bx_0$ by $\bx_n$ and $\bn_t$ by $\bu_n$, the argument holds verbatim.  The only assertion requiring additional comment is that $\bu_n\circ\bu/\bu_n\circ\bu_n\to 1/2$ as $n\to \infty$, used in the estimate on equation (\ref{check convex}) to show that the point of $P_t\cap U_{n+1}$ with maximal first entry lies in $B^-$.  Here we use Lemma \ref{intersect ballssss} to prove this: the ball $U_n$ of the lemma contains no points of $\cals$, so by cocompactness the radius of $U_n\cap P$ must remain bounded.

Therefore formula (\ref{check convex}) for $\bx\circ\bu$ approaches $0$ as $n\to\infty$, where $\bx\in P_n\cap\mathbb{H}^n$ has maximal first entry.  The claim follows, so for large enough $P_n^-\cap\cals$ is contained in the finite set $\cals_0$ of $\bs\in\cals$ with first entry at most $M+1$.  But $P_n^-\cap\cals$ is the vertex set of $F_n$, so there must exist $m\neq n$ such that $F_m$ and $F_n$ share a vertex set.  It therefore follows from Lemma \ref{rotate space} that $F_m$ and $F_n$ are identical, contradicting our hypothesis, so we have local finiteness.\end{proof}

\begin{corollary}\label{locally finite hyperbolic}  If $\cals$ is locally finite and invariant under a torsion-free lattice $\Gamma$ of $\mathit{SO}^+(1,n)$ then the collection $\{r_n(F)\,|\,F\ \mbox{is a visible face of}\ \mathrm{Hull}(\cals)\}$ is locally finite.\end{corollary}

\begin{proof}  Suppose this does not hold, and $K$ is a compact set in $\mathbb{H}^n$ intersecting $r_n(F_k)$ for a sequence $F_k$ of distinct faces of $\mathrm{Hull}(\cals)$.  Passing to a subsequence we assume that there exist $\bx_k\in F_k$ such that $\{r_n(\bx_k)\}\to \bx$ for some $\bx\in K$.

If for any $\bx_0\in r_n^{-1}(\bx)$ a subsequence of the $\bx_k$ converged to $\bx_0$, local finiteness of the collection of faces of $\mathrm{Hull}(\cals)$ (Lemma \ref{locally finite}) would be violated.  Therefore $\bx_k\circ\bx_k\to-\infty$ as $k\to\infty$.  Arguing as in the proof of Lemma \ref{convex hull boundary} produces a time-like $n$-dimensional subspace $V$ of $\mathbb{R}^{n+1}$ that has $\cals$ on one side.  This violates Lemma \ref{no time}, a contradiction.\end{proof}

\begin{theorem}\label{pseudo EP}\PseudoEP\end{theorem}

\begin{remark}\label{lattice geometric dual}  For $\cals$ as in Theorem \ref{pseudo EP}, the geometric dual complex of $\cals$ includes all Delaunay cells but the parabolic-invariant ones.  See Theorem \ref{the real geometric dual}.\end{remark}

\begin{proof}It is immediate that the empty circumspheres condition (2) above uniquely determines a collection of convex subsets of $\mathbb{H}^n$.  We will show that the Delaunay tessellation we  defined in \ref{faces n stuff} has this and the other listed properties.  By Definition \ref{faces n stuff} it is the collection:
$$\{r_n(F)\,|\, F\ \mbox{is a visible face of}\ \mathrm{Hull}(\cals)\}$$
By Lemma \ref{invariance'n'stuff} the collection of faces of $\mathrm{Hull}(\cals)$ is $\Gamma$-invariant, so since $r_n$ is $\mathit{SO}^+(1,n)$-invariant the Delaunay tessellation is as well.  Moreover, $\cals$ is locally finite so the prior results of this section apply to it.  Thus by Corollary \ref{locally finite hyperbolic} the collection of Delaunay cells is locally finite, and by Corollary \ref{all the points} their union is $\mathbb{H}^n$.

By Lemma \ref{no time}, a face $F$ of $\mathrm{Hull}(\cals)$ is of the form $P\cap\mathbb{H}^n$ for a support plane $P$ that is parallel to a space-like or light-like vector subspace of $\mathbb{R}^{n+1}$.  In the former case Lemma \ref{rotate space} implies that $F=\mathrm{Hull}(P\cap\cals)$ is a compact, convex polyhedron.  Therefore by Lemma \ref{poly to poly}, $r_n(F)$ is a compact, convex polyhedron in $\mathbb{H}^n$ equal to the closed convex hull of $S\cap\cals$ where $S = P\cap\mathbb{H}^n$.  If $P$ is parallel to a light-like subspace then by Proposition \ref{horospherical cells}, $r_n(F)$ is a convex polyhedron equal to the closed convex hull of $P\cap\cals$ in $\mathbb{H}^n$.

The proofs that each face of each cell is a cell, and that distinct cells that intersect do so in a face of each, follow those of the corresponding assertions of Proposition \ref{the finite case}.  The only modifications needed are to appeal to Lemma \ref{no time} instead of Lemma \ref{good visibility}.  This proves (1).


Lemma \ref{no time} implies that any support plane $P$ for $\mathrm{Hull}(\cals)$ separates $\cals$ from $\bo$. Thus for any face $F$ of $\mathrm{Hull}(\cals)$, $r_n(F)$ is contained in the convex region $B\cap\mathbb{H}^n$ bounded by $S = P\cap\mathbb{H}^n$ by Lemma \ref{ratatouille}, where $B$ is the half-space bounded by $P$ that contains $\bo$.  If $S = P\cap\mathbb{H}^n$ is a hypersphere such that $S\cap\cals\neq\emptyset$ but $B\cap\cals=S\cap\cals$, where $B$ is the half-space bounded by $P$ with $B\cap\mathbb{H}^n$ convex, then $P$ is a support plane for $\mathrm{Hull}(\cals)$ separating $\bo$ from $\mathrm{Hull}(\cals)$ (again by Lemma \ref{convex component}).  Hence $F = P\cap\mathrm{Hull}(\cals)$ is a visible face.  This proves the empty circumspheres condition (2).

For a support plane $P$ for $\mathrm{Hull}(\cals)$ parallel to a space-like subspace, $S = P\cap\mathbb{H}^n$ is a metric sphere by Lemma \ref{classification of spheres}.  $S$ is the circumsphere for $r_n(F)$, where $F = P\cap\mathbb{H}^n$, and $r_n(F)$ is compact as proved above.  If $P$ is parallel to a light-like subspace $V$ then by Proposition \ref{horospherical cells}, $C_U = r_n(F)$ is $n$-dimensional, the closed convex hull in $\mathbb{H}^n$ of $P\cap\cals$, and preserved by the (parabolic) stabilizer $\Gamma_U$ of $U=V^{\perp}$ in $\Gamma$.  This proves (3).\end{proof}

For the sake of clarity we will record the cocompact case separately.  It is a direct consequence of Theorem \ref{pseudo EP} and the standard fact that a cocompact lattice in $\mathit{SO}^+(1,n)$ contains no parabolics.  (This follows for example from Lemma \ref{invariant horoballs}.)

\begin{corollary}\label{cocompact case}  For a torsion-free, cocompact lattice $\Gamma<\mathit{SO}^+(1,n)$, the Delaunay tessellation of a non-empty $\Gamma$-invariant set $\cals\subset\mathbb{H}^n$ that has finite image under $\pi\co\mathbb{H}^n\to\mathbb{H}^n/\Gamma$ is a locally finite, $\Gamma$-invariant polyhedral decomposition of $\mathbb{H}^n$ that is the union of its $n$-cells.  For each metric sphere of $S$ of $\mathbb{H}^n$ that intersects $\cals$ and bounds an open ball disjoint from $\cals$, the closed convex hull of $S\cap\cals$ is a Delaunay cell.  Each Delaunay cell has this form.\end{corollary}

The cocompact case is in turn a special case of the \textit{co-finite} case described below.

\begin{corollary}\label{co-finite}  Let $\Gamma<\mathit{SO}^+(1,n)$ be a torsion-free lattice and $\cals\subset\mathbb{H}^n$ a non-empty $\Gamma$-invariant set with finite image under $\pi\co\mathbb{H}^n\to\mathbb{H}^n/\Gamma$.  Then $\cals$ is locally finite so the conclusions of Theorem \ref{pseudo EP} apply, and each parabolic fixed point $U$ of $\Gamma$ has a non-compact $\Gamma_U$-invariant $n$-cell as described in Theorem \ref{pseudo EP}(3).\end{corollary}

\begin{proof}  If $\cals$ is not locally finite it has an accumulation point $\bs_0\in\mathbb{H}^n$.  For $\pi$ as above, if $\pi(\cals)$ is finite then a sequence $\{\bs_n\}\subset\cals$ that approaches $\bs_0$ has a subsequence of $\Gamma$-equivalent points.  For any $\epsilon>0$ it follows that the closed ball $B$ about $\bs_0$ of radius $\epsilon$ intersects infinitely many of its $\Gamma$-translates.  Hence the $\Gamma$-action is not discontinuous, so $\Gamma$ is not discrete (see \cite[Theorem 5.3.5]{Ratcliffe}), a contradiction.

It remains to show that for each parabolic fixed point $U$ there is a support plane $P$ for $\mathrm{Hull}(\cals)$ separating $\cals$ from $\bo$ and parallel to $V = U^{\perp}$.   For then the half-space $B$ bounded by $P$ and containing $\bo$ intersects $\mathbb{H}^n$ in a horoball intersecting $\cals$ only in its boundary.  If $B\cap\cals$ is non-empty we may replace $P$ by a nearby parallel copy in $B$ and run the same argument to get a disjoint horoball.

Since $\pi(\cals)\subset \mathbb{H}^n/\Gamma$ is finite, $\pi(\cals)\subset M_{[\epsilon,\infty)}$ for some $\epsilon>0$.  We may assume that $\epsilon$ is less than the $n$-dimensional Margulis constant, so components of $M_{(0,\epsilon]}$ satisfy the classification Theorem D.3.3 of \cite{BenPet} sketched in Section \ref{margulis}.

Let $\Gamma_1$ be the parabolic subgroup of $\Gamma$ stabilizing $U$.  By Lemma \ref{invariant horoballs} there is a cusp $L$ of $M$ and a component $\widetilde{L}$ stabilized by $\Gamma_1$, and there exists $\bu\in U$ such that the horoball $\{\bx\circ\bu\geq-1\}\cap\mathbb{H}^n$ is contained in $\widetilde{L}$.  Then $\bs\circ\bu <-1$ for all $\bs\in\cals$, since $\cals$ lies outside $\widetilde{L}$.  Let $c_0 = \sup\{\bs\circ\bu\,|\,\bs\in\cals\}$ and let $\bu_0 = \bu/-c_0$.  Then $\sup\{\bs\circ\bu_0\,|\,\bs\in\cals\} = -1$.

The plane $P_0 = \{\bx\circ\bu_0 = -1\}$ separates $\bo$ from $\cals$.  We claim that $P\cap\cals\neq\emptyset$; i.e. that the supremum above is a maximum, whence $P$ is a support plane for $\mathrm{Hull}(\cals)$ and $F=P\cap\mathrm{Hull}(\cals)$ satisfies the claim above.  Lemma \ref{bad balls} implies for a sequence $\bs_k\in\cals$ with $\bs_k\circ\bu_0\to-1$ that $d_H(\bx_k,\bs_k)\to0$, where $\bx_k$ is the nearest point of $P_0\cap\mathbb{H}^n$ to $\bs_k$ for each $k$.  By $\Gamma$-invariance of $\cals$ we may assume that all $\bx_k$ lie in a compact fundamental domain for the $\Gamma_1$-action on $P_0\cap\mathbb{H}^n$.  Thus a subsequence of the $\bx_k$, hence also of the $\bs_k$ converges to some $\bx\in P_0\cap\mathbb{H}^n$, so $\bs_k = \bx$ for all large enough $k$ by local finiteness of $\cals$.\end{proof}

We finally describe the image in the quotient manifold in the setting of Theorem \ref{pseudo EP}.

\begin{corollary}\label{down below}  Let $\Gamma<\mathit{SO}^+(1,n)$ be a torsion-free lattice and $\cals$ a non-empty, locally finite, $\Gamma$-invariant set in $\mathbb{H}^n$.  The interior of each compact Delaunay cell embeds in $\mathbb{H}^n/\Gamma$ under $\pi\co\mathbb{H}^n\to\mathbb{H}^n/\Gamma$.  For a cell $C_U$ with parabolic stabilizer $\Gamma_U$, $\pi|_{\mathit{int}\,C_U}$ factors through an embedding of $\mathit{int}\,C_U/\Gamma_U$ to a set containing a cusp of $\mathbb{H}^n/\Gamma$.

If $\cals_0=\pi(\cals)$ is finite then there are finitely many $\Gamma$-orbits of Delaunay cells.
\end{corollary}

\begin{proof}  For any cell $F$ and $g\in\Gamma$, $g(F)$ is also a Delaunay cell (by $\Gamma$-invariance), so if $g(F)\neq F$ then $g(F)\cap \mathit{int}\,F=\emptyset$.  But since $\Gamma$ is torsion-free, if $F$ is compact then $g(F)\neq F$ for every $g\in\Gamma$ (see eg.~\cite[Cor.~2.8]{BrH}) so $\pi$ is embedding on $\mathit{int}\,F$.  If $F$ is parabolic-invariant then any $g\in\Gamma$ with $g(F) = F$ preserves the circum-horosphere of $F$, thus also its ideal point $U$, and therefore lies in $\Gamma_U$.  It follows that $\pi|_{\mathit{int}\,C_U}$ factors through an embedding of $\mathit{int}\,C_U/\Gamma_U$.

By Corollary \ref{its a poly} a parabolic-invariant cell $C_U$ contains a horoball centered at some $\bu\in U$, so the final assertion of Lemma \ref{invariant horoballs} implies $C_U/\Gamma_U$ contains a cusp of $\mathbb{H}^n/\Gamma$.

The empty circumspheres condition implies that the set of $0$-cells of the Delaunay tessellation is $\cals$.  The set of $\Gamma$-orbits of $\cals$ is in bijective correspondence with its image under $\pi$, so if $\pi(\cals)$ is finite then so is the set of orbits.  By local finiteness, each point of $\cals$ lies in only finitely many Delaunay cells, so since each such cell contains a point of $\cals$ they also have finitely many $\Gamma$-orbits in this case.\end{proof}

\bibliographystyle{plain}
\bibliography{Delaunay}

\end{document}